\documentclass[11pt,a4paper]{article}

\usepackage{mathtools,amssymb,amsthm,amsmath}
\usepackage[margin=1in]{geometry}
\numberwithin{equation}{section}

\usepackage{cite}
\usepackage{caption}
\usepackage{float}
\usepackage{subfloat}
\usepackage{comment}
\usepackage{multirow}
\usepackage{booktabs}
\usepackage{subcaption}
\usepackage{enumitem}% http://ctan.org/pkg/enumitem
\usepackage{url}
\usepackage{xcolor}
\usepackage{algorithmic}
\usepackage[hidelinks]{hyperref}
\hypersetup{
    colorlinks=true,
    linkcolor=blue,
    filecolor=blue,      
    urlcolor=blue,
    citecolor = blue,
    }
\usepackage[linesnumbered,algoruled,boxed,lined]{algorithm2e}
\usepackage{mathtools}
\usepackage{bm}
\usepackage{tabularx}
\newcolumntype{L}{>{\raggedright\arraybackslash}X}

\usepackage[pagewise]{lineno}

\def\RR{\mathbb R}
\def\EE{\mathcal F}
\def\WW{\mathcal{W}}

\def\BE{\mathbb E}
\def\PP{\mathcal P}

\def\e{\varepsilon}

\def\argmin{{\rm arg}\!\min}
\def\dt{\Delta t}
\newcommand{\defeq}{\mathrel{\mathop:}=}
\def\be{\begin{equation}}
\def\ee{\end{equation}}
\def\bea{\begin{eqnarray}}
\def\eea{\end{eqnarray}}

\def\Q{{{X}}}

\def\DD{{\mathcal D}}

\newcommand{\rev}[1]{{\color{black}#1}}
\newcommand{\tempconst}{\sigma_0}
\newtheorem{remark}{Remark}[section]
\newtheorem{assumption}{Assumption}[section]
\newtheorem{lemma}{Lemma}[section]
\newtheorem{proposition}{Proposition}[section]
\newtheorem{corollary}{Corollary}[section]
\newtheorem{theorem}{Theorem}[section]

\title{Swarm-based optimization with jumps: a kinetic BGK framework and convergence analysis}

\author{Giacomo Borghi\thanks{Corresponding author}$\,\,$\thanks{Maxwell Institute for Mathematical Sciences and Department of Mathematics, School of Mathematical and Computer Sciences (MACS), Heriot-Watt University, Edinburgh, UK
  (\texttt{g.borghi@hw.ac.uk}, \texttt{hi3001@hw.ac.uk}, \texttt{l.pareschi@hw.ac.uk}).}
  \and Hyesung Im\footnotemark[2]
\and Lorenzo Pareschi\footnotemark[2]$\,\,$\thanks{Department of Mathematics and Computer Science, University of Ferrara, Italy.}
}

\begin{document}

\maketitle

\begin{abstract}
Metaheuristic algorithms are powerful tools for global optimization, particularly for non-convex and non-differentiable problems where exact methods are often impractical. Particle-based optimization methods, inspired by swarm intelligence principles, have shown effectiveness due to their ability to balance exploration and exploitation within the search space. In this work, we introduce a novel particle-based optimization algorithm where velocities are updated via random jumps, a strategy commonly used to enhance stochastic exploration. We formalize this approach by describing the dynamics through a kinetic modelling of BGK type, offering a unified framework that accommodates general noise distributions, including heavy-tailed ones like Cauchy. Under suitable parameter scaling, the model reduces to the Consensus-Based Optimization (CBO) dynamics. For non-degenerate Gaussian noise in bounded domains, we prove propagation of chaos and convergence towards minimizers. Numerical results on benchmark problems validate the approach and highlight its connection to CBO.
\end{abstract}

\bigskip
\noindent 
\textbf{keywords:}
swarm-based optimization, BGK model, stochastic particle systems, propagation of chaos, global optimization

\bigskip
\noindent
% REQUIRED
\textbf{MSCcodes} 65K10, 90C26, 65C35, 82C40, 35Q90

\tableofcontents
%%%%%%%%%%%%%%%%%%%%%%%%%%%%%%%%%%%%%%%%%%%%%%%%%%%%%%%%%%%%%%%%%%%%%%%%%%%%
\section{Introduction}

Global optimization plays a critical role in solving complex, real-world problems across diverse fields such as machine learning, signal processing, optimal control, and finance. Among the most versatile approaches to global optimization are metaheuristic algorithms: high-level, gradient-free strategies inspired by natural phenomena or social behaviour. These methods are particularly effective in tackling large-scale and NP-hard optimisation problems where traditional exact methods become computationally infeasible. Classical examples include Simulated Annealing (SA) \cite{kirkpatrickOptimizationSimulatedAnnealing1983}, Particle Swarm Optimization (PSO) \cite{kennedyParticleSwarmOptimization1995}, Ant Colony Optimization (ACO) \cite{dorigoAntSystemOptimization1996}, and Genetic Algorithms (GA) \cite{hollandAdaptationNaturalArtificial1992}.

Metaheuristics are designed to navigate complex, often multimodal search spaces by balancing exploitation (deep search near good solutions) and diversification (broad exploration). They have proven effective in domains characterized by combinatorial complexity and uncertainty, including logistics \cite{cerulli2018computational}, finance \cite{doering2019finance}, manufacturing \cite{rossi2020integration}, and healthcare \cite{rubio2018multiobjective}. Modern extensions and hybridisations, such as combining with exact methods \cite{fischetti2018matheuristics}, simulation models  \cite{juan2018simheuristics}, and machine learning techniques \cite{calvet2017learnheuristics}, have further improved their ability to adapt to stochastic and dynamic environments.  For a broader perspective on the evolution and taxonomy of these approaches, we refer to the review  \cite{juan2023review}.

Due to their heuristic and application-driven design, the convergence analysis of metaheuristic algorithms has historically received limited attention. This is especially true for population-based methods, such as PSO, GA, and ACO, where a set of candidate solutions is iteratively mixed, perturbed, and updated through stochastic rules until a satisfactory result is found. Classical mathematical tools developed for the analysis of exact algorithms often prove inadequate in these contexts.

A more recent line of research, initiated with the development of the Consensus-Based Optimization (CBO) algorithm \cite{pinnauConsensusbasedModelGlobal2017}, seeks to establish rigorous mathematical foundations for metaheuristics using tools from statistical physics. The central idea is to model these algorithms as interacting particle systems whose evolution is governed by stochastic dynamics. The microscopic update rules for the candidate solutions, or particles, can often be described by PDEs similar to those found in physical systems, albeit typically with additional nonlinear or nonlocal interactions. This connection is conceptually intuitive, given that many metaheuristics are originally inspired by natural processes.

This PDE-based perspective has proven fruitful in providing analytical insights 
on the algorithms' mechanisms and convergence properties towards minimizers. In particular, in \cite{grassiParticleSwarmOptimization2021,huangGlobalConvergenceParticle2023}, the authors model a generalized version of PSO via a Vlasov--Fokker--Planck equation, which, under a suitable low-inertia limit \cite{ciprianiZeroinertiaLimitParticle2022} converges to the Fokker--Planck equation modelling CBO-type algorithms \cite{carrilloAnalyticalFrameworkConsensusbased2018,fornasierConsensusBasedOptimizationMethods2024}. GA algorithms, which are based on a microscopic binary dynamics simulating gene reproduction, have been modelled instead via kinetic equation of Boltzmann type \cite{albi2023kinetic,borghiKineticDescriptionConvergence2025}.
The case of the SA algorithm differs somewhat from the population-based approaches discussed above. Due to its natural connection with statistical physics and Markov Chain Monte Carlo methods, SA has long been studied through the lens of statistical mechanics, as in \cite{holley1988simulated}. More recently, a more PDE-based modelling approach based on linear Boltzmann equations was proposed in \cite{pareschi2024linear}.
Beyond offering a mathematical foundation, casting computational heuristics into PDE frameworks helps uncover structural similarities and differences across algorithms that are often obscured by their metaphor-driven and algorithm-specific terminology. We refer the reader to \cite{borghi2024kinetic} for a comprehensive review on this topic.

Following this line of research, this paper contributes to the development and the theoretical analysis of metaheuristics by: 
\begin{enumerate}[label=\roman*)]
\item 
designing and studying a novel particle-based optimization algorithm where particles explore promising regions of the search space following a second-order dynamics. In contrast to classical PSO approaches \cite{kennedyParticleSwarmOptimization1995}, particle velocities in our model are updated via random jumps, and no memory mechanisms or individual preferences are included.
Similar jump-based strategies have been explored in various swarm-based metaheuristics to address premature convergence and enhance stochastic exploration; see, for instance, \cite{krholing2004, krohling2005gaussian,krohling2009cauchy,li2007fast,mendel2011swarm,choi2023noise}. This algorithmic strategy also shares similarities with mutation operators in Genetic Algorithms (GAs) and with the jump-diffusion variant of the CBO algorithm \cite{kaliseConsensusbasedOptimizationJumpdiffusion2023}, although in those cases, the jumps affect the particle positions rather than their velocities.

\item 
We then derive a kinetic model of BGK type \cite{cercignaniBoltzmannEquationIts1988} that describes the dynamics in the mean-field (many-particle) limit, starting from a time-discretization of the particle system. A key feature of our modelling framework is its flexibility in accommodating various types of noise, unlike usual Fokker--Planck models, which typically assume Gaussian perturbations \cite{grassiParticleSwarmOptimization2021,huangGlobalConvergenceParticle2023}. For example, several popular metaheuristics employ Cauchy noise \cite{krohling2009cauchy,li2007fast,szu1987fast}, which allows for occasional long-range jumps. We also demonstrate how a CBO-type model can be obtained as a diffusive limit through a suitable rescaling of parameters.

\item
In the simplified setting of non-degenerate Gaussian noise and bounded search space, we provide quantitative estimates for the approximation error between the particle system and its kinetic model. This is achieved by proving a propagation of chaos result using a coupling argument in the $L^1$-Wasserstein distance. Furthermore, we identify sufficient conditions on the objective function under which the kinetic model converges to the global minimum, following the strategy proposed in \cite{fornasierConsensusBasedOptimizationMethods2024} for CBO methods. To the best of our knowledge, this is the first convergence analysis of this kind applied to second-order dynamics. We  conjecture that our proof technique could be extended to other second order PSO-type dynamics, improving upon the convergence analysis carried out in \cite{huangGlobalConvergenceParticle2023}.
\item 
Finally, we complement the study by testing the algorithm’s performance on benchmark problems across a range of parameter settings, with particular emphasis on the scaling regime that leads to the CBO algorithm.
\end{enumerate}

The paper is organized as follows: in Section \ref{sec:2} we present the optimization algorithm, formally derive the swarm-based model, and underline relations with CBO models. In Sections \ref{sec:chaos} and \ref{sec:convergence} we show quantitative propagation of chaos of the particle system and convergence towards global minimum for a specific model choice.
Numerical experiments on benchmark problems are presented in Section \ref{sec:numerics}, while final remarks and outlook are collected in Section \ref{sec:conclusion}.

%%%%%%%%%%%%%%%%%%%%%%%%%%%%%%%%%%%%%%%%%%%%%%%%%%%%%%%%%%%%%%%%%%%%%%%%%%%%
\section{Swarm-based optimization with jumps and BGK description}
\label{sec:2}

\subsection{Second order models with jump velocity update}

 We consider the following optimization problems of the type
\begin{equation}\label{typrob}
x^\ast \in \argmin\limits_{x\in \RR^d}\EE(x)\,,
\end{equation}
where $\EE:\mathbb R^{d} \to \mathbb R$ is a given objective function, which we wish to minimize over the search space $\RR^d, d\in \mathbb{N}$. Particle Swarm Optimization is a particle-based algorithm inspired by the collective motion of a flock of birds \cite{kennedyParticleSwarmOptimization1995}. At a high-level description, the standard PSO algorithm is characterized by the following elements:

\begin{itemize}[label=$\bullet$]
\item Particles follow a second-order dynamics, and so each particle is characterized by its position and its velocity;
\item The particles' velocities are updated according to a global and a local alignment. The global alignment directs the velocity towards the best solution found by the particle swarm. The local one directs each particle towards its own best solution found;
\item The velocity updates are subject to noise to improve exploration of the search space.
\end{itemize}
From the PSO algorithm proposed in \cite{kennedyParticleSwarmOptimization1995}, throughout the years many different variants have been then proposed to solve different type of optimization problems beyond \eqref{typrob}, see e.g. \cite{pso2022survey} for a review. 

We take as a starting point for our swarm-based optimization algorithm, the PSO version proposed in \cite{grassiParticleSwarmOptimization2021} where the position of the best particle is regularized, and the stochastic mechanism is decoupled from the deterministic one. This version is more amenable to PDE-modelling and theoretical analysis. We also consider the dynamics without local alignment, and, so, without memory mechanisms. 
At each algorithmic step $k = 0,1,2,\dots$ we consider $N$ particles $X^i_k, i = 1,\dots,N$ with velocities $V_k^i, i = 1,\dots, N$. Let $\Delta t>0$ be a step-size and $\lambda, \sigma>0$ be two parameters that govern the deterministic and stochastic components, respectively. Starting from a random initialization $(X^i_0,V_0^i)\in \RR^d\times \RR^d$ the particle updates read for $i = 1,\dots,N$
\begin{equation}\label{eq:pso}
\begin{split}
    X^i_{k+1}&= X^i_k+\dt V^i_{k+1}, \\
    V^i_{k+1}&= V^i_{k} + \lambda\Delta t (X^\alpha[\rho_k^N] - X^i_k) + \sigma \sqrt{\Delta t} (X^\alpha[\rho_k^N] - X^i_k) \odot \xi^i_k 
\end{split}
\end{equation}
where $\xi^i_k\in \RR^d$ are randomly sampled vectors from a distribution component-wise symmetric around $0$, and $\odot$ is the component-wise product. Above, the point $X^\alpha[\rho_k^N]$ is defined as a weighted average of the particles' positions
\begin{equation} \label{eq:xalpha}
    X^\alpha[\rho_k^N] = \frac{\sum_{i=1}^NX^i_k\omega_\alpha(X^i_k)}{\sum_{i=1}^N \omega_\alpha(X^i_k)}, \quad \omega_\alpha(x)=\text{exp}(-\alpha \EE(x)), \quad \alpha>0\,,
\end{equation}
which depends on the particles' current positions through the empirical probability distribution  $\rho_k^N = 1/N \sum_i \delta_{X^i_k}$ associated with the system. The point $X^\alpha[\rho_k^N]$ can be considered to be a regularization of the best solutions among the particles' positions, as it holds
\[
X^\alpha[\rho_k^N] \longrightarrow \underset{x  =  X^i_k, i = 1,\dots,N}{\textup{argmin}}\,\EE(x) \qquad \textup{as}\quad \alpha \to \infty
\]
provided there is only one "best" particle among the ensemble. This regularization was proposed in the context of CBO \cite{pinnauConsensusbasedModelGlobal2017} , and therefore we call it the "consensus point". Indeed, the particle dynamics \eqref{eq:pso} can also be seen as a second-order CBO update. After either a fixed number of iterations or the application of an early stopping criterion, the final consensus point is taken as the desired minimizer.

In \eqref{eq:pso}, the velocities are incrementally updated by adding terms proportional to the step-size $\Delta t$. In the literature, 
different algorithmic strategies have been proposed \cite{krholing2004,krohling2009cauchy,li2007fast} where velocities are instead completely changed from the previous iterations, and sampled anew.  This is usually done in combination with the usual velocity update of type \eqref{eq:pso} to occasionally increase the exploration of the particles and prevent premature convergence. We propose the following swarm-based algorithm \eqref{eq:pso} , which includes velocity jumps, while being amenable to mathematical modelling.

Let $\nu>0$ be a parameter determining the frequency of the velocity jumps, the modified dynamics reads for $i = 1, \dots,N$ 
\begin{equation}\label{eq:psojump}
\begin{split}
    X^i_{k+1}&= X^i_k+\dt V^i_{k+1}, \\
    V^i_{k+1}&= 
    \begin{cases}
    V^i_k & \textup{with prob.}\;\; e^{-\nu \Delta t}\,, \\
    \lambda (X^\alpha[\rho_k^N] - X^i_k) + \sigma (X^\alpha[\rho_k^N] - X^i_k) \odot \xi_k^i & \textup{with prob.}\;\; 1- e^{-\nu \Delta t}\,.
    \end{cases}
\end{split}
\end{equation}
 While uniform distributions are the standard choice for the noise $\xi^i_k$ in PSO algorithms \cite{pso2022survey}, variants with Gaussian or Chauchy random vectors are also popular, see e.g. \cite{krholing2004,krohling2009cauchy,li2007fast}.
In \eqref{eq:psojump}, $\Delta t $ enters in the velocity dynamics via the update frequency, rather as a proper step-size as in \eqref{eq:pso}. Specifically, large values of $\Delta t$ lead to more frequent updates, and the same holds for the parameter $\nu$. We remark that the system is not overparametrized, as $\Delta t$ also controls the step-size of increments for the particles positions.

\begin{remark}
In both \eqref{eq:pso} and \eqref{eq:psojump}, the strength of the noise depends on the difference $X^\alpha[\rho^N_k] - X_k^i$ , so it disappears as consensus among the particles emerges, that is $X^i_k \to X^\alpha[\rho_k^N]$. This is also a feature of CBO-type algorithms. In our theoretical analysis, we will consider a more general model accounting for non-degenerate diffusion (see update \eqref{eq:particle-sys-alg}). The benefit of non-degenerate diffusion in the theoretical analysis of CBO was already studied in \cite{huangFaithfulGlobalConvergence2025} where it allows for uniqueness of steady state solutions of the corresponding mean-field model. In high-dimensional applications, it has been used to prevent early convergence towards sub-optimal solutions, see, for instance,  \cite[Remark 2.4]{carrilloConsensusbasedGlobalOptimization2020} and the numerical experiments therein.
    \end{remark}

\subsection{A kinetic BGK description of the algorithm}

In this section, we propose a mathematical model for the evolution of a swarm of particles with jumps, as defined in \eqref{eq:psojump}, by taking the many-particle limit ($N \to \infty$) and infinitesimal time steps ($\Delta t \to 0$). Unlike the modeling approaches for standard PSO \cite{grassiParticleSwarmOptimization2021} and CBO \cite{carrilloAnalyticalFrameworkConsensusbased2018,fornasierConsensusBasedOptimizationMethods2024}, we first pass to the limit $N \to \infty$ to obtain a kinetic finite-difference equation, and only then consider its continuous-time limit. In this section, we present formal derivations; more rigorous and quantitative results on the kinetic approximation will be given in Section \ref{sec:chaos}.

To model the particle dynamics, we adopt the classical propagation of chaos assumption on the marginals \cite{sznitman1991chaos,chaintronPropagationChaosReview2022}. Let the particle states $\{X_k^i, V_k^i\}_{i=1}^N$ be random variables in $\mathbb{R}^d \times \mathbb{R}^d$, with joint law $F_k \in \mathcal{P}((\mathbb{R}^d \times \mathbb{R}^d)^N)$. If the particles are initially independent and identically distributed with common law $f_0 \in \mathcal{P}(\mathbb{R}^d \times \mathbb{R}^d)$, we assume that this independence propagates over time despite the interactions. That is, for $k \geq 1$ and $N \gg 1$,
\[
F_k \approx f_k^{\otimes N} \qquad \text{for some} \quad f_k \in \mathcal{P}(\mathbb{R}^d \times \mathbb{R}^d)\, .
\]
We extend the definition of the consensus point to any $\rho \in \mathcal{P}(\mathbb{R}^d)$ as
\[
X^\alpha[\rho] := \frac{\int x \, \omega_\alpha(x)\, \rho(dx)}{\int \omega_\alpha(x)\, \rho(dx)},
\]
with $\omega_\alpha$ defined as in \eqref{eq:xalpha}. Under the propagation of chaos assumption, we have $X^\alpha[\rho_k^N] \approx X^\alpha[\rho_k]$ for $N \gg 1$, where $\rho_k$ is the first marginal of $f_k$.

Replacing $X^\alpha[\rho_k^N]$ with $X^\alpha[\rho_k]$ in \eqref{eq:psojump}, we obtain a system of independent particles evolving according to the following nonlinear update:
\begin{equation}\label{eq:monopso}
\begin{split}
    \overline{X}_{k+1} &= \overline{X}_k + \Delta t\, \overline{V}_{k+1}, \\
    \overline{V}_{k+1} &= 
    \begin{cases}
    \overline{V}_k & \text{with prob.}\;\; e^{-\nu \Delta t}, \\
    \lambda (X^\alpha[\rho_k] - \overline{X}_k) + \sigma (X^\alpha[\rho_k] - \overline{X}_k) \odot \xi_k & \text{with prob.}\;\; 1 - e^{-\nu \Delta t},
    \end{cases} \\
    \rho_k &= \text{Law}(\overline{X}_k),
\end{split}
\end{equation}
with $\text{Law}(\overline{X}_0, \overline{V}_0) = f_0$. We show next that the above mono-particle process is related to a numerical discretization of a kinetic PDE of Bhatnagar–Gross–Krook (BGK) type.

For notational simplicity, we assume that $f_k$ admits a density, also denoted $f_k$.  Let $f = f(x,v)$ be any probability density, with $\rho(x) = \int f(x,v) dv$ its first marginal. Consider $x\in \RR^d$ and any $d$-dimensional random variable $\xi_k$, we define 
$\mathcal{M}^x[\rho] := \text{Law}\left(\lambda (X^\alpha[\rho] - x) + \sigma (X^\alpha[\rho_k] - x) \odot \xi_k\right)$
and $\mathcal{M}[f](x,v) := \rho(x) \mathcal{M}^x[\rho](v)$. Let us define the operators $Q^{\Delta t}, L^{\Delta t}$ acting on $f$ as
\be
\begin{split}
    Q^{\Delta t} f(x,v) &:= e^{-\nu\Delta t}f(x,v) +(1-e^{-\nu\Delta t}) \mathcal{M}[f](x,v) \\
    L^{\Delta t} f(x,v) &:= f(x - \Delta t v, v) \,.
\end{split}
\label{eq:discretePDE}
\ee

\begin{lemma}\label{lem:consistency}
If $\textup{Law}(\overline{X}_k, \overline{V}_k) = f_k$, then $\textup{Law} (\overline{X}_{k+1},\overline{V}_{k+1}) =  L^{\Delta t}Q^{\Delta t}f_k\,.$
\end{lemma}
\begin{proof}

Let $\overline{W}_k = \lambda (X^\alpha[\rho_k] - \overline{X}_k) + \sigma (X^\alpha[\rho_k] - \overline{X}_k) \odot \xi_k$, for any test function $\phi\in\mathcal{C}_b(\RR^d\times \RR^d)$ it holds
\begin{equation*}
\begin{split}
\mathbb{E} \left[ \phi (\overline{X}_k, \overline{V}_{k+1}) \right]& = e^{-\nu\Delta t} \mathbb{E} [\phi(\overline{X}_k, \overline{V}_k)] + (1-e^{-\nu\Delta t})\mathbb{E} [\phi(\overline{X}_k, \overline{W}_k)]  \\ 
&= e^{- \nu\Delta t}\int \phi(x,v) f_k(x,v)dxdv + (1-e^{-\nu\Delta t})\int \phi(x,v) \rho_k(x) \mathcal{M}^x[\rho_k](v)dxdv \\
&= \int \phi(x,v) Q^{\Delta t}f_k(x,v)dxdv\,.
\end{split}
\end{equation*}
Therefore, for $f_{k+\frac12}:=\textup{Law}(\overline{X}_k, \overline{V}_{k+1}) $ we have $f_{k+\frac12} = Q^{\Delta t}f_k $.
Next, we note that
\begin{equation*}
\begin{split}
    \mathbb{E} \left[\phi(\overline{X}_{k+1}, \overline{V}_{k+1})\right] &= \mathbb{E} \left[\phi(\overline{X}_k+\Delta t \overline{V}_{k+1}, \overline{V}_{k+1})\right] \\
    &= \int \phi(x+v\Delta t,v) f_{k+{1\over2}}(x, v)dx dv \\
    &=\int \phi(x,v) L^{\Delta t} f_{k+1}(x,v) dx dv\,,
\end{split}
\end{equation*}
which shows that $f_{k+1} = L^{\Delta t} f_{k+\frac12} = L^{\Delta t } Q^{\Delta t }f_k $.
\end{proof}

%We consider the discretization of the PDE \eqref{BGK1} using a splitting method that separates the transport and relaxation steps. For a theoretical analysis of particle systems approximating the discretized PDE via splitting methods, we refer the reader to \cite{desvillettesSPLITTINGALGORITHMBOLTZMANN1996}. 

Updates \eqref{eq:discretePDE} correspond to a splitting scheme widely used in the numerical approximation of kinetic equations, see \cite{desvillettesSPLITTINGALGORITHMBOLTZMANN1996}. In the continuous-time limit $\Delta t \to 0$, they approximate the BGK-type equation:
\begin{equation}\label{BGK1}
\partial_t f + v \cdot \nabla_x f = \nu (\mathcal{M}[f] - f),
\end{equation}
where $\mathcal{M}(x,v,t) = \rho(x,t) \mathcal{M}^x[\rho](v)$.
Equation \eqref{BGK1} is a classical BGK model, originally introduced in \cite{bhatnagarModelCollisionProcesses1954}, used as a simplified alternative to the Boltzmann equation in rarefied gas dynamics. In such systems, particles move with given velocities and interact through collisions that drive the distribution toward a local Maxwellian equilibrium $\mathcal{M}[f]$. The BGK formulation replaces the complex collision term in the Boltzmann equation with a relaxation term that is more amenable to analysis and computation. In \eqref{BGK1}, the relaxation towards $\mathcal{M}[f]$ corresponds to the operator $Q^{\Delta t}$, while the free transport corresponds to the action of $L^{\Delta t}$.

In \cite{buttaStochasticParticleSystem2022,buttaParticleSystemsBGK2023}, the authors showed that stochastic particle systems approximate the BGK equation on a torus in dimensions $d=2,3$, providing further support for interpreting the swarm-based dynamics with jumps \eqref{eq:psojump} as a Monte Carlo approximation of the BGK model \eqref{BGK1}.

Unlike classical BGK models, where the Maxwellian equilibrium is typically Gaussian in velocity, the equilibrium in our BGK-swarm model depends on the noise distribution $\xi_k$ in \eqref{eq:monopso}. Assuming $\xi_k$ is standard Gaussian, $\xi_k \sim \mathcal{N}(0,I_d)$, we obtain:
\begin{equation}\label{eq:maxwellian}
    \mathcal{M}^x[\rho] = \mathcal{N}\left(\lambda (X^\alpha[\rho] - x),\; \sigma^2 \, \text{diag}(X^\alpha[\rho] - x)^2\right).
\end{equation}
This BGK system conserves mass but not momentum or energy, which are instead shaped by the Maxwellian structure and the driving consensus term $X^\alpha[\rho]$ that directs the system toward the estimated global minimum.

\begin{remark}In \cite{grassiParticleSwarmOptimization2021,pinnauConsensusbasedModelGlobal2017,kaliseConsensusbasedOptimizationJumpdiffusion2023}, the authors provided a kinetic description of the algorithms by first describing the particle dynamics as a system of SDEs as $\Delta t \to 0$, and then, by taking the mean-field limit $N\to 0$. This could also be done for the swarm dynamics with jumps \eqref{eq:psojump} by defining suitable non-linear Markov generators and Poisson processes \cite{applebaumLevyProcessesStochastic2013}. While we conjecture the limits $ \rev{\dt \to 0}, N\to \infty$ to be interchangeable, we leave this alternative modelling procedure for future research.

\end{remark}

\subsection{Diffusion limit and relation to CBO}\label{sec:2.scaling}

To explore the connection between the BGK optimization model \eqref{BGK1} and CBO on particle swarms, let us
consider Gaussian local Maxwellians \eqref{eq:maxwellian}, and introduce the following diffusive scaling
\[
t \to t/\e^2,\quad x \to x/\e,\quad \lambda \to \e\lambda, 
\] 
where $\e > 0$ is a scaling term.
The particle density $f(x,v,t)$, $x,v \in \RR^d$, $d \geq 1$ evolves accordingly to the scaled BGK system 
\begin{equation}\label{BGK}
\begin{split}
\partial_t f + \frac1{\e}v \cdot \nabla_x f = \frac{\nu}{\e^2}({\mathcal M}_\e-f)
\end{split}
\end{equation}
where $\e > 0$ is a scaling factor, and the density of the scaled local Maxwellian is explicitly given, for $x = (x_1, \dots, x_d)^\top, v = (v_1, \dots, v_d)^\top \in \RR^d$ by
\be
\begin{split}
{\mathcal M}_\e(x,v,t)&= \rho(x,t) \prod_{\ell=1}^d M_\e(x_\ell,v_\ell,t), \\
M_\e(x_\ell,v_\ell,t) &= \frac{1}{\pi^{1/2}\sigma |x_\ell-\Q^\alpha_\ell[\rho]|} 
\exp\left\{-\frac{(v_\ell-\e\lambda(\Q^{\alpha}_\ell[\rho]-x_\ell))^2}{\sigma^2(x_\ell-\Q^\alpha_\ell[\rho])^2}\right\},
\end{split}
\ee
where now, each one-dimensional Maxwellian $M_\e(x_\ell,v_\ell,t)$ has vanishing momentum as $\e \to 0$.

Let us now integrate equation \eqref{BGK} with respect to $v$, and multiply the same equation by $v$ and integrate again, we get
\be
\begin{split}
\frac{\partial \rho}{\partial t} + \nabla_x \cdot J_\e &= 0\\
\frac{\partial J_\e}{\partial t} + \frac{1}{\e^2}\int_{\RR^d} v\left(v\cdot \nabla_x f\right)\,dv &= \frac{\nu}{\e^2} (\lambda (\Q^{\alpha}[\rho]-x)\rho -J_\e) 
\end{split}
\label{macro}
\ee
where we used the scaled moment
\[
J_\e = \frac1{\e}\int_{\RR^d} f(x,v,t) v\,dv.
\]
Clearly system \eqref{macro} is not closed since it involves the knowledge of higher order moments of the kinetic equation \eqref{BGK}.
As $\e \to 0$ from the right hand side in \eqref{BGK} we get
\be
f(x,v,t) = {\mathcal M}_0(x,v,t),
\ee
the Maxwellian with zero momentum, and then the $\ell$-th component of the second term in the left hand side of
the second equation in \eqref{macro} becomes
\[
\begin{split}
\int_{\RR^d} v_\ell\left(v\cdot \nabla_x \left({\mathcal M}_0(x,v,t)\right) \right)\,dv &=  \sum_{j=1}^d \frac{\partial}{\partial x_j} \left(\int_{\RR^d} v_\ell v_j\, {\mathcal M}_0(x,v,t)\,dv\right)\\
%& = \e^2\lambda^2 \sum_{i \neq j} \frac{\partial}{\partial x_i} \left[\rho(x,t)  (x_i-\Q^\alpha_i(\rho))(x_j-\Q^\alpha_j(\rho))\right]\\ 
& =  \frac{\partial}{\partial x_\ell} \left(\rho(x,t) \int_{\RR^d} v^2_\ell {M}(x_\ell,v_\ell,t)\,dv_\ell\right)\\
%& = \e^2\lambda^2 \sum_{i \neq j} \frac{\partial}{\partial x_i} \left[\rho(x,t) (x_i-\Q^\alpha_i(\rho))(x_j-\Q^\alpha_j(\rho))\right]\\
& = \frac{\sigma^2}{2}\frac{\partial}{\partial x_\ell} \left(\rho(x,t) (x_\ell-\Q^\alpha_\ell[\rho])^2\right).
\end{split}
\]
Therefore, in the limit $\e\to 0$, from the second equation in \eqref{macro} we have 
\[
(J_0)_\ell = \lambda (\Q_\ell^{\alpha}[\rho]-x_\ell) \rho -\frac{\sigma^2}{2} \frac{\partial}{\partial x_\ell} \left(\rho(x,t) (x_\ell-\Q^\alpha_\ell[\rho])^2\right),
\]
which substituted in the first equation yields the mean-field CBO system \cite{carrilloConsensusbasedGlobalOptimization2020}
\begin{equation}
\frac{\partial \rho}{\partial t} + \nabla_x \cdot \lambda (\Q^{\alpha}[\rho]-x) \rho = \frac{\sigma^2}{2}\sum_{\ell=1}^d \frac{\partial^2}{\partial x^2_\ell} \left(\rho(x,t) (x_\ell-\Q^\alpha_\ell[\rho])^2\right).
\label{eq:CBOp}
\end{equation}
Hence, in the diffusion limit $\e \ll 1$ we expect the macroscopic density in the kinetic system \eqref{BGK} to be well approximated by the solution of the CBO equation \eqref{eq:CBOp}. 
%We remark that this is not the case for the original CBO method proposed in \cite{Pinnau_2017} where the noise is not in component-wise form. 

%%%%%%%%%%%%%%%%%%%%%%%%%%%%%%%%%%%%%%%%%%%%%%%%%%%%%%%%%%%%%%%%%%%%%%%%%%%%
We remark that the swarm-based optimization with jump update \eqref{eq:psojump} applied to the scaled system \eqref{BGK} leads to the update rules
\begin{equation}\label{alg_BGK2}
\begin{split}
    X^i_{k+1}&= X^i_k+\dt V^{\e, i}_{k+1}, \\
    V^{\e,i}_{k+1}&= 
    \begin{cases}
    V^{\e,i}_k & \textup{with prob.}\;\; e^{-\nu \Delta t/ \e^2}\,, \\
    \lambda (X^\alpha[\rho_k^N] - X^i_k) + \frac\sigma\e (X^\alpha[\rho_k^N] - X^i_k) \odot \xi_k^i & \textup{with prob.}\;\; 1- e^{\rev{-\nu\dt / \e ^2}}\,.
    \end{cases}
\end{split}
\end{equation}

%\begin{equation}
%\begin{split}
%X_i^{n+1} &= X_i^n + \Delta t V_{\e,_i}^{n+1}\\
%V_{\e,i}^{n+1} &= \Psi(\xi_i < e^{-\nu\Delta t/\e^2}) V_{\e,i}^n + (1-\Psi(\xi_i < e^{-\nu\Delta t/\e^2})) \left(\lambda (X_{\alpha}^n-X^n_i)+\eta_i \frac{\sigma}{\e} |X_{\alpha}^n-X^n_i| \right)
%\end{split}
%\label{alg_BGK2}
%\end{equation}  
which is unsuitable to approximate the system behaviour for small values of $\varepsilon$. 
%A better algorithm which is consistent in the diffusion limit is derived following the ideas in \cite{DPS18}. 

\begin{remark} Note that the above idea applies to any Fokker--Planck system of the form
\begin{equation}
\frac{\partial \rho}{\partial t} + \nabla_x \cdot \left(A[\rho]\rho\right) = \frac{\sigma^2}{2}\sum_{\ell=1}^d \frac{\partial^2}{\partial x^2_\ell} \left(\rho(x,t) D[\rho]_\ell^2\right),
\label{eq:FP}
\end{equation}
using the one-dimensional Maxwellian components 
\[
M_{\e}(x_\ell,v_\ell,t) = \frac{1}{\pi^{1/2}\sigma |D[\rho]_\ell|} 
\exp\left\{-\frac{(v_\ell-\e A[\rho]_\ell)^2}{\sigma^2 D[\rho]^2_\ell}\right\},
\]
and a BGK kinetic model of the form \eqref{BGK}.
\end{remark}

\begin{remark}[Fluid limit]
In the classical fluid-limit $\nu\to \infty$ we have $f=\mathcal{M}$ and integrating \eqref{BGK1} with respect to $v$ we get 
\be
\begin{split}
\frac{\partial \rho}{\partial t} + \nabla_x \cdot  \lambda \left(X^{\alpha}[\rho]-x\right)\rho &= 0.
\end{split}
\label{euler}
\ee
From an algorithmic viewpoint, when $\nu\to\infty$ we have $e^{-\nu\Delta t} \to 0$ leading to the reduced dynamic 
\be
X^i_{k+1} = X^i_k + \Delta t \left(\lambda (X^{\alpha}[\rho_k]-X^i_k) + \sigma (X^\alpha[\rho_k^N] - X^i_k) \odot \xi_k^i \right),
\label{aCBO}
\ee
that corresponds to a consistent stochastic particle approximation to \eqref{euler} characterized by a Maxwellian travelling in space with a velocity and a variance which depends on the particle position.  Note that in \eqref{aCBO} the noise vanishes as $\Delta t \to 0$ and we get the ODE
\[
\frac{d}{dt} X^i = \lambda (X_{\alpha}-X^i)\,.
\]
\end{remark}

%%%%%%%%%%%%%%%%%%%%%%%%%%%%%%%%%%%%%%%%%%%%%%%%%%%%%%%%%%%%%%%%%%%%%%%%%%%%
\section{Propagation of chaos of the particle system}

\label{sec:chaos}

In this section we quantitatively study the relation between the swarm-based optimization with jumps, and the related non-linear mono-particle process obtained formally in Section \ref{sec:2} under the propagation of chaos assumption. Specifically, we aim to demonstrate that, as the number of particles $N \to \infty$ the law of the $N$-particle system converges to that of the corresponding mean-field dynamics and each particle process becomes asymptotically independent. We will consider a modified version of the dynamics where particles explore a bounded search domain and are subject to non-degenerate diffusion.

This convergence property is commonly referred to as propagation of chaos. With the results of propagation of chaos, we can justify using mean-field equation as a particle system in further analysis. Our proof is based on a coupling method, which is one of several established techniques used to derive such results. For a comprehensive overview of the various methods and developments in the theory of propagation of chaos, we refer the reader to \cite{chaintronPropagationChaosReview2022}.

\subsection{Notations and preliminaries}
\label{sec:notation}

For $x=(x_1, \dots, x_d)^\top \in \RR^d$, we denote with $|x|$ the standard Euclidean norm, while $\lVert x \rVert_1 := \sum_{\ell=1}^d|x_\ell|$, $\lVert x \rVert _\infty:=\max_{\ell=1,\dots,d} |x_\ell|$. With $B_r(x)$ and $B_r^\infty(x)$ we denote the closed balls centred in $x$ with radius $r>0$ for the Euclidean and $\ell^\infty$ norm, respectively.

If a random variable $X$ follows a Bernoulli distribution with parameter $p\in[0,1]$, we write $X \sim \text{Bern}(p)$. If $X$ follows a Gaussian distribution with mean $m \in \RR^d$ and covariance matrix $\Sigma \in \RR^{d\times d}$, we write $X \sim \mathcal{N}(m, \Sigma)$.
Let $\mathcal{P}(\mathcal{X})$ be the set of probability measures on a Polish space $\mathcal{X}$, and $\mathcal{P}_p(\mathcal{X})$ the set of probability measures with finite moments of order $p$, that is, $M_p(\mu) := \left(\int |x|^p \mu(dx) \right) ^{\frac{1}{p}} < \infty$. For a measure $\mu \in\mathcal{P}(\mathcal{X})$ and a Borel map $T:\mathcal{X}\to \mathcal{Y}$, we denote by $T_{\#}\mu \in \mathcal{P}(\mathcal{Y})$ the push-forward of $\mu$ through $T$, defined by 
 \begin{equation*}
     T_{\#}\mu(B) \defeq \mu(T^{-1}(B)) \quad \forall B \in \mathcal{B}(\mathcal{Y}),
 \end{equation*}
 where $\mathcal{B}(\cdot)$ denotes Borel sets on the space. 
%polish space: complete, separable metric spaces%

For two probability measures $\mu, \nu \in \mathcal{P}_p(\RR^d)$,  $\WW_p(\mu, \nu)$ is the $p$-Wasserstein distance defined by 
\begin{equation*}
    \WW_p(\mu, \nu) \defeq \min_{\gamma \in \Gamma} \left(  \int |x-y|^p \gamma(dx, dy )\right)^{\frac{1}{p}},
\end{equation*}
with 
 $\Gamma = \{ \gamma \in \mathcal{P}(\RR^d \times \RR^d) \,| \,\pi_{0\#}\gamma =\mu, \pi_{1\#}\gamma = \nu\}$ and $\mu, \nu \in \mathcal{P}_p(\RR^d)$. Here, $\pi_0(x, y) = x$ and $\pi_1(x, y) = y$ are the canonical projections onto the first and second coordinates. 
 For more detailed definitions and properties, we refer the reader to \cite{santambrogioOptimalTransportApplied2015,villaniOptimalTransport2009}.

\subsection{Assumptions and main results}

As mentioned in the introduction, we consider a modified version of the particle update \eqref{eq:psojump} accounting for projection in a closed convex search domain $\DD\subseteq \RR^d$, and non-degenerate Gaussian diffusion. The projection map $\Pi_{\DD}:\RR^d \to \DD$ is given by
\[
\Pi_{\DD}(y):= \underset{x\in\DD}{\textup{argmin}}\, | x - y|\,,
\]
and it is well defined thanks to the convexity of $\DD$. As before, consider $N$ particles $(X^i_k, V^i_k), i = 1,\dots,N$, where $k = 0,1,2,\dots$ denotes the algorithmic step, and the corresponding empirical measures $f_k^N = (1/N) = \sum_i \delta_{(X^i_k,V^i_K)}$ and $\rho^N_{k} = (1/N)\sum_i \delta_{X_k^i}$. For a given step size $\Delta t>0$ and a probability measure $p^\xi\in \mathcal{P}(\RR^d)$ for the noise, at every algorithmic step define i.i.d. random variables
\[
T^i_k \sim \textup{Bern}(e^{-\nu\Delta t }), \quad \rev{\xi^i_k} \sim \mathcal{N}(0,I_d)\qquad i = 1,\dots,N
\]
and for $\lambda, \sigma> 0, \sigma_0 \geq 0$ we define $W_k^i$ component-wise for $\ell = 1,\dots,d$ as
\begin{equation} \label{eq:Wi}
(W_k^i)_\ell = \lambda(X^\alpha[\rho_k^N] - X_k^i)_\ell + \sigma\left(\sigma_0 + |(X^\alpha[\rho_k^N] - X_k^i)_\ell|\right) \xi_{k,\ell}^i
\qquad i = 1,\dots,N
\,. 
\end{equation}
The particle update is then given by
\be
\begin{cases}
    X_{k+1}^i &= \Pi_{\DD} \left(X_k^i + \Delta t V_{k+1}^i\right) \\
    V_{k+1}^i &= T_k^i V_k ^i + (1-T_k^i) W_k^i 
\end{cases} \qquad i = 1,\dots,N\,.
\label{eq:particle-sys-alg}
\ee
We note that for $\DD = \RR^d$ and $\sigma_0  = 0$, the update is equivalent to \eqref{eq:psojump}. For $\sigma_0>0$, the contribution of the noise $\xi_k^i$ never vanishes in the dynamics. As similar CBO dynamics with non-degenerate diffusion was studied in \cite{huangFaithfulGlobalConvergence2025}.

Following the formal arguments of Section \ref{sec:2}, one obtains that the corresponding nonlinear particle system of is given by
\be
\begin{cases}
    \overline{X}_{k+1} &= \Pi_{\DD} \left(\overline{X}_k + \Delta t \overline{V}_{k+1} \right) \\
    \overline{V}_{k+1} &= \overline{T}_k \overline{V}_k + (1-\overline{T}_k)\overline{W}_k
\end{cases}
\label{eq:nonlinear-update}
\ee
where, similarly as before, $ \overline{T}_k \sim \text{Bern} \ (e^{-\nu\Delta t})$, $\overline{\xi}_k \sim \mathcal{N}(0,I_d)$, and, as before, for $\ell = 1,\dots,D$ we have
\begin{equation} \label{eq:Wbar}
    (\overline{W}_k)_\ell = \lambda(X^\alpha[\rho_k] - X_k)_\ell + \sigma\left(\sigma_0 + |(X^\alpha[\rho_k] - \overline{X}_k)_\ell|\right) \overline{\xi}_{k,\ell}\,.
\end{equation}

The system is of McKean type, as the consensus point $X^\alpha[\rho_k]$ depends on $\rho_k = \textup{Law}(\overline{X}_k)$.
% Assumptions on objective function
Throughout the analysis, we assume the following conditions for the objective function $\mathcal{F}$. 
\begin{assumption}\label{assume:propagation}The objective function $\mathcal{F}: \RR^d \to \RR$ satisfies the following:
\begin{itemize}
    \item [(A1)] $\mathcal{F}$ is bounded below with $\underline{\mathcal{F}}\defeq \inf \mathcal{F} > -\infty$. 
    \item [(A2)] (growth conditions) there exists $L_{\mathcal{F}}, c_u, c_l, R_l>0$ such that \\  
    \be
    \begin{cases}
            |\mathcal{F}(x)-\mathcal{F}(y)| \leq L_\mathcal{F}(1+|x|+|y|)|x-y| &\forall x,y \in \RR^d,  \\
            \mathcal{F}(x)-\mathcal{F}(x^\ast) \leq c_u(1+|x|^2) &\forall x \in \RR^d, \\
            \mathcal{F}(x)-\mathcal{F}(x^\ast) \geq c_l|x|^2  &\forall x \text{ s.t. } |x| > R_l
        \end{cases}
    \ee
\end{itemize}
\end{assumption}

\begin{theorem} \label{thm:propagation-of-chaos}
Let $\mathcal{D}\subset \RR^d$ be a compact and convex search space, and let $\{ (X^i_k, V^i_k)\}_{i=1}^N$ be a particle system defined by \eqref{eq:particle-sys-alg} with i.i.d. initial data $(X^i_0, V^i_0)\sim f_0$ for some $f_0\in \mathcal{P}_1(\DD\times \RR^d)$.

Consider $N$ independent copies $\{(\overline{X}^i_k,\overline{V}^i_k)\}_{i=1}^N$ be $N$ of  the mono-particle process \eqref{eq:nonlinear-update} 
with $\overline{T}^i_k = T_k^i \sim \textup{Bern}(e^{-\nu \dt})$, $\overline{\xi}_k^i = \xi_k^i\sim \mathcal{N}(0,I_d)$ for all $i =1,\dots,N$, and same initial data.  Let the parameters be such that $\Delta t \in (0,1]$, $\lambda, \sigma>0$, $\sigma_0\geq 0$ and let $T > 0$ be a fixed time horizon.

If the objective function $\EE$ satisfies Assumption \ref{assume:propagation}, then there exists a constant $C= C(\alpha,\DD,\lambda, \sigma, \mathcal{F},d)$ independent of $N$, $T$ such that
\begin{equation}\label{eq:propagation-of-chaos}
  \sup_{k\Delta t \in [0,T]}\BE \left[ \frac{1}{N} \sum_{i=1}^N |X^i_k-\overline{X}^i_k| + |V^i_k-\overline{V}^i_k| \right] \leq Ce^{CT}\frac{1}{\sqrt{N}}. 
\end{equation}
\end{theorem}

Since the auxiliary nonlinear particles $\{(\overline{X}^i_k,\overline{V}^i_k)\}_{i=1}^N$ are independent and identically distributed, the above result can be extended to estimate the distance between the empirical measure $f_k^N$ and the kinetic model $f_k$.

\begin{corollary} \label{cor:propagation-of-chaos} 
Under the same settings of Theorem \ref{thm:propagation-of-chaos},  further assume that $f_0 \in \mathcal{P}_q(\mathcal{D}\times\RR^d )$ with $q>2d/(2d-1)$.
Then, there exists a positive constant $C=C(\mathcal{D},d,q,f_0,\alpha,\lambda, \sigma,\sigma_0 ,\mathcal{F})$ independent on $T$ and $N$ such that 
\begin{equation*}
 \sup_{k\Delta t\in [0,T]}\BE [\WW_1(f^N_k, f_k)] \leq C\left( e^{CT}\frac1{\sqrt{N}} +   \e(N)\right)
\end{equation*}
where 
\be \label{eq:eps1}
\e(N):=
 \begin{cases}
N^{-1/2}\log(1+N)  & \textup{if} \;\;d = 1 \\
N^{-1/(2d)} & \textup{if} \;\;d>1.
\end{cases}
\ee
\end{corollary}

\subsection{Proof of Theorem~\ref{thm:propagation-of-chaos}}

The proof is based on a coupling between the two particle systems, where the track the distance between the $i$-th particles $(X_k^i, V_k^i)$ and $(\overline{X}^i_k, \overline{V}_k^i)$ at each time step $k$. To do so, we first collect some stability results for the consensus point $X^\alpha[\cdot]$. 

\begin{lemma} \label{lem:measure-measure}
    Let $\mathcal{F}$ satisfy Assumption \ref{assume:propagation} and $\mu, \hat{\mu} \in \mathcal{P} (\RR^d)$ with $$ \int |x|^4 \mu(dx), \quad \int |x|^4 \hat{\mu}(dx) \leq K.$$ Then we have the following estimates
    $$|X^\alpha[\mu]-X^\alpha[\hat{\mu}]| \leq C_0 \WW_1(\mu, \hat{\mu}), $$
    where $C_0$ is a positive constant depending only on $\alpha, L_\EE, K$.
\end{lemma}
\begin{proof}
We note that the estimate is equivalent to the one derived in  \cite[Lemma 3.2]{carrilloAnalyticalFrameworkConsensusbased2018} with $\WW_1$ instead of $\WW_2$. The proof can be carried out with the exact same argument, but without applying Jensen's inequality in the last step.
\end{proof}

\begin{lemma}[\cite{fornasierConsensusbasedOptimizationHypersurfaces2021}, Lemma 3.1] \label{lem:measure-law}
    Let $\EE$ satisfy Assumption \ref{assume:propagation} and $\DD\subseteq \RR^d $ be closed compact domain. Let $\{ \overline{X}^i_k\}_{i=1}^N$ for $k=0, 1, \dots $ be i.i.d. with common distribution $\rho_k \in \mathcal{P}(\DD)$ and denote with $\overline{\rho}^N_k$ be the corresponding empirical measure. 
    
   Then, there exists a constant $C_1$ depending only on $\textup{diam}(\mathcal{D)}$ and $C_{\alpha, \EE} \defeq \exp(\alpha (\sup_{x\in \DD}\EE(x) - \inf_{x\in \DD}\EE(x)))$
    such that 
    $$ \sup_{k \in \mathbb{N}} \BE \left[ |X^\alpha [\overline{\rho}^N_k]-X^\alpha [\rho_k]|\right] \leq C_1 N^{-1 / 2}.$$
\end{lemma}
\begin{proof} The proof can be carried out exactly as in \cite[Lemma 3.1]{fornasierConsensusbasedOptimizationHypersurfaces2021} where it is shown that $\sup_{k \in \mathbb{N}} \BE \left[ |X^\alpha [\overline{\rho}^N_k]-X^\alpha [\rho_k]|^2\right] \leq \tilde{C}_1 N^{-1}$. Also in this case, our claim simply follows from Jensen's inequality $\mathbb{E}[|X|]\leq \sqrt{\mathbb{E}[|X|^2]}$ for any random variable $X$.
\end{proof}
%%%%%%%%%%%%%%%%%%%%%%%%%%%%%%%%%%%%%%%%%%%%%%%%%%%%%%
\begin{proof}[Proof of Theorem \ref{thm:propagation-of-chaos}]

First, we note that the projection map is non-expansive 
\[
|\Pi_\DD(x) = \Pi(y)| \leq |x- y| \quad\textup{for any}\quad x,y \in \RR^d\,,
\]
as a consequence of the convexity of the domain $\DD$. Therefore, at any algorithmic step $k$, and for each $i=1, \dots ,N$ we have
\be\label{eq:position}
\begin{split}
\mathbb{E}\left[|X^i_{k+1}-\overline{X}^i_{k+1}|\right ] &= \mathbb{E}\left[\left|\Pi_\mathcal{D} \left( X^i_k + \Delta t V^i_{k+1} \right)-\Pi_\mathcal{D} \big( \overline{X}^i_k + \Delta t \overline{V}^i_{k+1} \big)\right|\right] \\
&\leq \mathbb{E}\left[|X^i_k-\overline{X}^i_k|+\Delta t |V^i_{k+1}-\overline{V}^i_{k+1}|\right].
\end{split}
\ee

For the velocities, thanks to the choice $\overline{T}_k^i = T^i_k$ and $T^i_k\sim \textup{Bern}(e^{-\nu \Delta t})$, we observe that
\be
\begin{split}
    \BE \left[ |V^i_{k+1}-\overline{V}^i_{k+1}| \right] &= \BE \left[ |T^i_k\overline{V}^i_k + (1-T^i_k) \overline{W}^i_k-T^i_k\overline{V}^i_k - (1-T^i_k) \overline{W}^i_k| \right] \\
    &\leq \BE \left[ T^i_k |V^i_k-\overline{V}^i_k| \right]+\BE \left[ (1-T^i_k)|W^i_k-\overline{W}^i_k|\right] \\
    &= e^{-\nu\Delta t}\BE \left [|V^i_k-\overline{V}^i_k| \right]+(1-e^{-\nu\Delta t})\BE \left[ |W^i_k-\overline{W}^i_k| \right].
\end{split}
\ee
where in the last step we used that $\mathbb{E}[T^i_k] = e^{-\nu \Delta t}$. By using the explicit definitions of $W_k^i, \overline{W}_k^i$, we can estimate their difference as 
\begin{equation*}
\begin{split}
    \BE \left[ |W^i_k-\overline{W}^i_k| \right] &= \BE \left[ |\lambda (X^\alpha [\rho^N_k]-X^\alpha [\rho_k] -X^i_k- \overline{X}^i_k)| \right]\\
    &\qquad + \BE \left[ \sqrt{\sum_{\ell=1}^s \sigma^2 ( \xi^i_{k})_\ell^2 \left(\sigma_0 +(|(X^\alpha [\rho^N_k] - X_k^i)_\ell|-\sigma_0 -|(X^\alpha [\rho_k] - \overline{X}^i_k)_\ell|\right)^2 }\right]\\
    &\leq  \lambda\BE \left[ |(X^\alpha [\rho^N_k]-X^\alpha [\rho_k] -(X^i_k- \overline{X}^i_k))|\right] \\ &\qquad +\sigma\sum_{\ell=1}^d \BE \left[ |(\xi^i_k)_\ell|\right]\BE\left[|(X^\alpha [\rho^N_k]-X^\alpha [\rho_k] -X^i_k- \overline{X}^i_k)_\ell|\right]  \\
    &\leq (\lambda+\sigma \sqrt{d})\BE \left[ |X^\alpha [\rho^N_k]-X^\alpha [\rho_k] -X^i_k- \overline{X}^i_k|\right] \\
    &\leq (\lambda+\sigma  \sqrt{d}) \Big( \underbrace{\BE \left[|X^\alpha [\rho^N_k]-X^\alpha [\rho_k]|\right]}_{\eqqcolon I} + \BE \left[|X^i_k- \overline{X}^i_k|\right]\Big),
\end{split}
\end{equation*}
where we used the norm inequalities 
\begin{equation}\label{eq:norms}
    |x|=\sqrt{\sum_{l=1}^d |x_l|^2} \leq \sum_{l=1}^d|x_l|=\lVert x \rVert_1 \leq \sqrt{d}|x|,
\end{equation}
and that $\rev{\BE[|(\xi^i_k)_l|]=\sqrt{2/\pi} \leq 1}$ with the assumption $\xi^i_k=\overline{\xi}_k$ for all $i=1, \dots ,N$.  We note that the estimate is completely independent on the additional diffusion term $\tempconst \geq 0$.

To estimate the term $I$ above, consider now the empirical measures $\overline{\rho}_k^N = (1/N)\sum_{i=1}^N \delta_{\overline{X}_k^i}$ associated with the particles' positions $\{\overline{X}_k^i\}_{i=1}^N$. The triangle inequality gives us 
$$\BE \left[ |X^\alpha [\rho^N_k]-X^\alpha[\rho_k]| \right] \leq \BE \left[ |X^\alpha [\rho^N_k]-X^\alpha[\overline{\rho}^N_k]| \right] + \BE \left[|X^\alpha [\overline{\rho}^N_k]-X^\alpha[\rho_k]| \right]$$
Then, thanks to the boundedness of $\DD$, we can apply Lemma \ref{lem:measure-measure} to obtain
$$\BE \left[ |X^\alpha [\rho^N_k]-X^\alpha[\overline{\rho}^N_k]| \right] \leq C_0 \BE \left[ \WW_1(\rho^N_k, \overline{\rho}^N_k) \right] \leq C_0 \BE \left[\frac{1}{N}\sum_{j=1}^N|X^j_k-\overline{X}^j_k|\right].$$
Note that the second inequality follows from the fact that $(X_k^i,\overline{X}^i_k)$ is only one of the possible couplings between the two particle system, and it is not necessarily the optimal one realizing the Wasserstein distance.
Meanwhile, Lemma \ref{lem:measure-law} yields
$$\BE \left[ |X^\alpha[\overline{\rho}^N_k]-X^\alpha [\rho_k]| \right] \leq C_1N^{-1/2}, $$ since $\{\overline{X}^i_k \}_{i=1}^N$ are i.i.d. with a common distribution $\rho_k.$ So far, the constants $C_0, C_1$ depend only on the size of the domain $\mathcal{D}, \alpha$, and the objective function $\mathcal{F}$. 

By collecting the above estimates, we can bound the expected difference in the velocities as 
\be \label{eq:velocity}
\begin{split}
    \BE \left[ |V^i_{k+1}-\overline{V}^i_{k+1}| \right]
    &\leq e^{-\nu\Delta t}\BE \left[|V^i_k-\overline{V}^i_k|\right] \\
    &+(1-e^{-\nu\Delta t})(\lambda+\sigma \sqrt{d}) \left( C_0 \BE \left[\frac{1}{N}\sum_{j=1}^N|X^j_k-\overline{X}^j_k|\right] + C_1N^{-1/2}  \right).
\end{split}
\ee
Combining the upper bounds \eqref{eq:position} and \eqref{eq:velocity} we obtain
{
\begin{equation*}
\begin{split}
    & \BE\left[ |X^i_{k+1}-\overline{X}^i_{k+1}|\right] + \BE \left[|V^i_{k+1}-\overline{V}^i_{k+1}|\right] \\
    &\leq \BE\left[ |X^i_k-\overline{X}^i_k| \right] + (1+\dt)\BE \left[|V^i_{k+1}-\overline{V}^i_{k+1}|\right]\\
    &\leq \BE\left[ |X^i_k-\overline{X}^i_k| \right] (1+\Delta t) e^{-\nu\Delta t} \BE \left[|V^i_k-\overline{V}^i_k|\right]  \\
    & \qquad+ (1-e^{-\nu\Delta t})(\lambda+\sigma \sqrt{d}) \left( C_0 \BE \left[\frac{1}{N}\sum_{j=1}^N|X^j_k-\overline{X}^j_k|\right] + C_1N^{-1/2}  \right) \\
    &= \BE \left[ |X^i_k-\overline{X}^i_k| \right]+C_0(1+\Delta t)\rev{(1-e^{-\nu \dt})} (\lambda +\sigma\sqrt{d} )\BE\left[\frac{1}{N}\sum_{j=1}^N|X^j_k-\overline{X}^j_k|\right]  \\ 
    &\qquad + (1+\Delta t)e^{-\nu\Delta t}\BE\left[ |V^i_k-\overline{V}^i_k| \right] +C_1(1+\Delta t)(1-e^{-\nu\Delta t}) (\lambda +\sigma  \sqrt{d} )N^{-1/2}\,,
\end{split}
\end{equation*}}
where in the last step we simply rearranged the terms. To later apply a (discrete) Gr{\"o}nwall-type argument, it is crucial to obtain estimates of order $(1 + C_2\Delta t)$,  or $C_3\Delta t$ for some constants $C_2,C_3$. As a consequence of $1-e^{-\nu \Delta t}\leq \nu \Delta t $, $e^{-\nu \Delta t} \leq 1$, and the assumption $\Delta t \in (0,1]$, we have the following further estimates:
\begin{align*}
(1 + \Delta t ) (1 - e^{-\nu \Delta t}) &\leq (1 + \Delta t) \nu \Delta t \leq \rev{2 \nu \dt} \\
e^{-\nu\Delta t} (1 + \Delta t) &\leq (1 + \Delta t)\,.
\end{align*}
With the above, and summing over all particles $i=1, \dots, N$, we get 
\begin{align*}
    \BE& \left[ \frac{1}{N} \sum_{i=1}^N |X^i_{k+1}-\overline{X}^i_{k+1}| + |V^i_{k+1}-\overline{V}^i_{k+1}| \right] \\
    &\leq (1+\Delta tC_2) \BE\left[ \frac{1}{N} \sum_{i=1}^N |X^i_k-\overline{X}^i_k| + |V^i_k-\overline{V}^i_k| \right] + \Delta t C_3 N^{-1/2}, 
\end{align*}
where $C_2$  and $C_3$ are constants depending on $d, \lambda, \sigma, \alpha, \text{diam}(\mathcal{D}), \mathcal{F}$ (from $C_0, C_1$). 
Iterating the above inequality for $h=1,\dots,k$, one obtains 
\begin{align*}
    \BE& \left[ \frac{1}{N} \sum_{i=1}^N |X^i_k-\overline{X}^i_k| + |V^i_k-\overline{V}^i_k| \right] \\ 
    &\leq (1+\Delta t C_2)^k \BE  \left[ \frac{1}{N} \sum_{i=1}^N |X^i_0-\overline{X}^i_0| + |V^i_0-\overline{V}^i_0| \right] + \Delta tC_3 N^{-1/2}\sum_{h=1}^k(1+\Delta t C_2)^h.
\end{align*}
Since $X^i_0 = \overline{X}^i_0, V^i_0 = \overline{V}^i_0$ for all $i=1, \dots, N$, the first term disappears. 
By using the formula $\sum_{h=1}^k(1+\Delta t C_2)^h = ((1+\Delta t C_2)^{k}-(1+\Delta tC_2))/(\Delta t C_2)\,$,
the second term can instead be simplified as
\begin{align*}
    \BE \left[ \frac{1}{N} \sum_{i=1}^N |X^i_k-\overline{X}^i_k| + |V^i_k-\overline{V}^i_k| \right]
    & \leq \Delta tC_3 \frac{1}{\sqrt{N}}\sum_{h=1}^k(1+\Delta t C_2)^h \\
    & = \frac{C_3}{C_2} \rev{\frac{1}{\sqrt{N}}} \left ( (1 +\Delta t C_2)^k - (1 +\Delta t C_2) \right) \leq \rev{\frac{C_3}{C_2\sqrt{N}}e^{k\Delta t C_2} }
\end{align*}
where we used $(1 +\Delta t C_2)^k \leq e^{k\Delta t C_2}$. Finally, taking supremum over $k\dt \in[0,T]$, we obtain \eqref{eq:propagation-of-chaos}.

\end{proof}

\subsection{Proof of Corollary~\ref{cor:propagation-of-chaos}}

With Theorem \ref{thm:propagation-of-chaos}, we quantified the distance between the particle system \eqref{eq:particle-sys-alg} and $N$ copies of independent, $f_k$-distributed, mono-particles processes \eqref{eq:nonlinear-update}. The derived estimate can be translated into an estimate on the 
Wasserstein distance between the corresponding empirical measure $f_k^N$ and $\overline{f}_k^N$, since
\begin{equation} \label{eq:coupling}
\WW_1(f_k^N, \overline{f}^N_k) \leq \frac{1}{N} \sum_{i=1}^N \left(|X_k^i-\overline{X}^i_k| +  |V_k^i-\overline{V}^i_k| \right).
\end{equation}
Note that we have used the fact that  $|(x,v) - (\overline{x},\overline{v})|\leq |x - \overline{x}| + |v - \overline{v}|$ for any $x,\overline{x},v,\overline{v}\in \RR^d$.

To quantify the distance between $\overline{f}^N_k$ and $f_k$, we recall the following quantitative result. Below, $\mu^N$ is the empirical measure of a system of independent $\mu$-distributed particles.
\begin{theorem}[{\cite[Theorem 1]{fournierRateConvergenceWasserstein2013}}] \label{lem:moments}
Let $\mu \in \mathcal{P}(\RR^{\tilde d})$ and let $p>0$. Assume that $M_q(\mu)<\infty$ for some $q>p$.  There exists a constant $C$ depending only on $p,{\tilde d},q$ such that for all $N \geq 1$:
$
\mathbb{E}\left[  \WW_p(\mu,\mu^N)\right] \leq C M_q^{p/q}(\mu) \e_{p}(N)
$
with 
\begin{equation*}
\e_{p}(N):=
\begin{cases}
N^{-1/2} + N^{-(q-p)/q} & \textup{if} \;\;p>{\tilde d}/2 \;\;\;\textup{and} \;\; q \neq 2p, \\
N^{-1/2}\log(1+N) + N^{-(q-p)/q} & \textup{if} \;\;p={\tilde d}/2 \;\;\;\textup{and} \;\; q \neq 2p, \\
N^{-p/{\tilde d}} + N^{-(q-p)/q} & \textup{if} \;\;p\in(0,{\tilde d}/2) \;\;\;\textup{and} \;\; q \neq {\tilde d}/({\tilde d}-p).
\end{cases}
\end{equation*}

\end{theorem}

Note that the order of convergence given in Corollary \ref{cor:propagation-of-chaos} is a consequence of the above theorem with $\tilde{d} = 2d$ (since we are considering the position-velocity space $\RR^d\times \RR^d$) and $p=1$.
Also, it tells us that the error introduced by any Monte Carlo strategy is related to the moments of the kinetic density. Therefore, we collect an estimate on the moments of $f_k = \textup{Law}(\overline{X}_k, \overline{V}_k)$, for which the assumption of boundedness of $\DD$ is crucial.

\begin{lemma}
\label{l:qmoments}
Let $f_0 \in \mathcal{P}_q(\DD\times\RR^d)$ for some $q>1$, and $f_k = \textup{Law}(\overline{X}_k, \overline{V}_k)$ with $(\overline{X}_k, \overline{V}_k)$ iteratively defined by \eqref{eq:nonlinear-update}.
There exists a positive constant $C = C(q,\mathcal{D},\lambda, \sigma,\sigma_0)$ such that
\[
M_q(f_k) \leq \max \left\{ M_q(f_0), C \right\} \qquad \textup{for all steps $k$.} 
\]
\end{lemma}
\begin{proof}
First, we note that if the particles' positions are projected towards the bounded domain $\mathcal{D}$ at every iteration, then the consensus point also belongs to $\DD$. Moreover, we have $\textup{ess} \sup|\overline{X}_k| \leq R$ for $R$ large enough such that $\DD \subset B(0,R)$. It follows $\BE[|\overline{X}_k|^q] \leq R^q$.
For the velocity, it holds 
\begin{align*}
\mathbb{E}[|\overline{V}_{k+1}|^q] &\leq e^{-\nu\Delta t} \mathbb{E}[|\overline{V}_{k}|^q] + (1 - e^{-\nu\Delta t })\mathbb{E}[|\overline{W}_{k}|^q] \\
& \leq \max\{\mathbb{E}[|\overline{V}_{k}|^q], \mathbb{E}[|\overline{W}_{k}|^q]\}\,.
\end{align*}
To estimate $\BE[|\overline{W}_{k}|^q]$, we again use that $|X^\alpha[\rho_k] - \overline{X}_k|\leq \textup{diam}(\DD)$ everywhere:
\begin{align*}
\BE[|\overline{W}_k|^q] &= \iint |\lambda(X^\alpha[\rho_k]) - x) + \sigma(\sigma_0 + X^\alpha[\rho_k] - x) |^q \rho_{k}(dx) p^\xi(d\xi) \\
& \leq c_q \int (\lambda \textup{diam}(\DD))^q + \sigma^q(\sigma_0 + \textup{diam}(\DD))^q |\xi|^q p^\xi(d\xi)\\
& \leq C_1
\end{align*}
where $c_q>0$ is a constant depending on $q$, while $C_1$ depends on $q,\mathcal{D},\lambda, \sigma,\sigma_0$ and the $q$-th moment of the standard Gaussian distribution. Therefore, since $\mathbb{E}[|\overline{X}|^q] \leq R^q$ and $\BE[|\overline{V}_k|^q] \leq \max\{\BE[|\overline{V}_0|],C_1^q\}$ we also have that the $q$-th moment $M_q(f_k)$ can be upper bounded with either $M_q(f_0)$ or a constant which depends on $C_1$.
\end{proof}

\begin{proof}[Proof of Corollary \ref{cor:propagation-of-chaos}.]
First, thanks to Lemma \ref{lem:moments} and Lemma \ref{l:qmoments}, we have that for all $k$
\[
 \BE \left[ \WW_1 (\overline{f}^N_k, f_k) \right] \leq CM_q(f_k)\e(N) \leq C_1 \e(N)
\]
with $C_1 = C_1(\mathcal{D},q,f_0, \lambda, \sigma, \sigma_0, p^\xi)>0$. Next, we note that by Theorem \ref{thm:propagation-of-chaos},
\[
\sup_{k\Delta t\in [0,T]} \mathbb{E} \left[\WW_1(f_k^N,\overline{f}_k^N) \right] \leq C_2e^{C_2T}\frac1{\sqrt{N}}
\]
where $C_2 = C_2(\alpha, \mathcal{D},\lambda, \sigma, \sigma_0 ,\mathcal{F})>0$. By triangular inequality, we can conclude that
\begin{align*}
\sup_{k\Delta t\in [0,T]} \BE \left[\WW_1(f^N_k, f_k)\right]&\leq
\sup_{k\Delta t\in [0,T]} \BE \left[\WW_1(f^N_k, \overline{f}^N_k)\right] + \sup_{k\Delta t\in [0,T]} \BE \left[ \WW_1(\overline{f}^N_k, f_k) \right]\\
&\leq C\left( e^{CT}\frac1{\sqrt{N}} +   \e(N)\right) \,.
\end{align*}
\end{proof}
%%%%%%%%%%%%%%%%%%%%%%%%%%%%%%%%%%%%%%%%%%%%%%%%%%%%%%%%%%%%%%%%%%%%%%%%%%%%

\section{Convergence to global minimum}
\label{sec:convergence}

In this section, we study under which conditions the kinetic approximation of the swarm-based optimization with jumps algorithm converges to a global minimum of the objective function. In line with the previous section, we consider the kinetic approximation \eqref{eq:nonlinear-update} to the particle dynamics \eqref{eq:particle-sys-alg} with projection towards a compact domain $\DD$ and non-generate diffusion.

The convergence analysis of a PSO-type dynamics without jumps was carried out in \cite{huangGlobalConvergenceParticle2023}, in time-continuous settings, following the approach proposed in \cite{carrilloAnalyticalFrameworkConsensusbased2018} for the analysis of CBO dynamics. The strategy relies, first, on proving an exponential decay of the variance of the system, which leads to convergence towards a specific point in the search space. Then, the authors provide an estimate on how far that point is from global minimizers.

In our analysis, we will follow a different strategy, proposed in \cite{fornasierConsensusBasedOptimizationMethods2024} for CBO dynamics, that provides non-asymptotic quantitative error estimates. The proof is based on a quantitative version of the Laplace principle (see Proposition \ref{prop:quant-laplace}) and an estimate on the average distance from the minimizer. For the first time, this approach is extended to second-order dynamics, like PSO ones. We will consider the optimization problem to attain a unique global solution $x^*$ (see Assumption \ref{assume:convergence} below), and track the error evolution through  the following functional
\begin{equation}\label{eq:energy}
\mathcal{H}[f] \defeq \int \left( \gamma\vert x-x^\ast \vert +  \vert v-\lambda(x^\ast-x) \vert\right)f(dx,dv),
\end{equation}
for any $f\in \mathcal{P}_1(\RR^d\times\RR^d)$. The weight $\gamma>0$ balances the contribution of the error $|x - x^*|$ in the position space and the one in the velocity space $\vert v-\lambda(x^\ast-x) \vert$. Note that $\mathcal{H}[f] = 0$ if and only if $f = \delta_{(x^*,0)}$, that is, if particles are concentrated on the minimizer with null velocity. We note that the adaptation of the strategy of \cite{fornasierConsensusBasedOptimizationMethods2024} to time-discrete settings was also done in \cite{borghi2024thesis} for first-order CBO-type dynamics, and that the case of non-degenerate diffusion was recently considered in \cite{huangFaithfulGlobalConvergence2025} , which allowed for stronger convergence results of CBO-type dynamics.

\subsection{Assumptions and main results}

For the notation of this section, we refer to the one previously introduced in Section \ref{sec:notation}. For the convergence analysis, we consider the non-linear mono-particle process $(\overline{X}_k,\overline{V}_k)$, $k = 0,1, 2,\dots$, updated according to \eqref{eq:nonlinear-update}, and its law $f_k =\textup{Law}(\overline{X}_k,\overline{V}_k)$.

In the previous section, we showed that if the objective function $\EE$ satisfies Assumption~\ref{assume:propagation}, 
then  $f_k$ can be considered to be a good approximation of the particle system for $N \gg1$. To show that convergence, in the sense that $\mathcal{H}[f_k] \ll 1$, we will consider additional assumptions on $\EE$:

\begin{assumption} \label{assume:convergence} 
    The objective function $\mathcal{F}: \RR^d \to \RR$ is continuous and satisfies the following:
    \begin{itemize}
        \item[(A1)]  (uniqueness) there exists a unique global minimizer $x^\ast \in \RR^d$;

        \item[(A2)]  (growth conditions around minimizer) there exists $c_p, p>0, R_p>0$ and a lower bound $\mathcal{F}_\infty$ such that 
        \be
        \begin{dcases}
            \|x-x^\ast\|^p_\infty \leq c_p \left( \mathcal{F}(x) -\mathcal{F}(x^\ast)\right) &\textup{for all} \;\;x \in B_r^\infty(R_p) \\
             \mathcal{F}_\infty< \mathcal{F}(x)-\mathcal{F}(x^\ast) & \textup{for all} \;\;x \in \RR^d \setminus B_r^\infty(R_p)\,.
        \end{dcases}
        \ee
    \end{itemize}
\end{assumption}

These assumptions are the same as the ones considered in \cite{fornasierConvergenceAnisotropicConsensusBased2022}. Assumption A2 is an inverse continuity  assumption  which locally requires a polynomial growth of $\EE$ in  a neighbourhood of the minimizers.  

%The Assumption A2 is required in order to guarantee that the mass around the minimizer $x^\ast$ does not vanish during the evolution, which will allow us to use the quantitative Laplace principle. 

\begin{theorem}[Main convergence theorem] \label{thm:convergence} 
Let $\mathcal{F}$ satisfy Assumption \ref{assume:convergence}, and $(\overline{X}_0, \overline{V}_0)$ be distributed according to a given $f_0 \in \mathcal{P}(\DD \times \RR^d)$ with $\DD$ compact domain and $\textup{supp}(f_0) = \DD\times \RR^d$. 
Let  $f_k = \textup{Law}(\overline{X}_k, \overline{V}_k)$ be updated according to \eqref{eq:nonlinear-update} with noise distribution $\overline{\xi}_k\sim \mathcal{N}(0,I_d)$.

Fix an arbitrary accuracy $\e>0$. Then, there exists a set of parameters $\lambda, \sigma,\sigma_0,\nu,\dt>0$ and $\gamma>0$ such that given the time horizon 
\begin{equation}
    T^\star \defeq \frac{2}{A}\log \left(\frac{2\mathcal{H}[f_0]}{\e} \right)\,,
\end{equation}
where $A$ is a parameter-dependent constant, it holds
\begin{equation}
    \min_{k\Delta t \in [0,T^\star]}\, \mathcal{H}[f_k] \leq \e
\end{equation}
provided $\alpha>0$ is sufficiently large. Moreover, it holds 
$\mathcal{H}[f_k]  \leq e^{-k A\Delta t/2}\mathcal{H}[f_0]$, 
until the desired accuracy $\e$ is reached.
\end{theorem}

\subsection{Proof of Theorem~\ref{thm:convergence}}

As mentioned, the central tool in the convergence analysis consists of estimating the distance between the consensus point and the global minimizers $x^*$ via a quantitative Laplace principle. We recall for completeness the result, with the notation of Assumption \ref{assume:convergence}.

\begin{proposition}[{\cite[Proposition 1]{fornasierConvergenceAnisotropicConsensusBased2022}}]
    Let $\rho \in \PP(\RR^d)$, $\underline{\mathcal{F}}=0$ w.l.o.g., and fix $\alpha>0$. Define $\mathcal{F}_r \defeq \sup_{x \in B_r(x^\ast)}\mathcal{F}(x)$. Then, under the Assumption \ref{assume:convergence}, for any $r \in (0, R_p]$ and $q>0$ such that $q+\mathcal{F}_r\leq \mathcal{F}_\infty$, we have 
    \begin{equation}\label{eq:quant-laplace}
         \vert X^\alpha[\rho]-x^\ast\vert \leq c_p\sqrt{d}(q+\mathcal{F}_r)^{1/p}+\frac{\sqrt{d}\exp(-\alpha q)}{\rho(B^\infty_r(x^\ast))}\int\vert x-x^\ast \vert  \rho(dx)\,.
    \end{equation}
\end{proposition}\label{prop:quant-laplace}

Note that we are using anisotropic noise version of the estimate, hence the $\sqrt{d}$ term is included in the upper bound, contrary to the isotropic version proposed in \cite{fornasierConsensusBasedOptimizationMethods2024}. 
As is clear from the estimate above, applying the quantitative Laplace principle requires a lower bound on the mass near the minimizer. In the following lemma, we establish an auxiliary result that provides control over the mass in velocity space, which is instrumental in deriving a lower bound for $\rho_k^\infty(B_r(x^\ast))$.

\begin{lemma} \label{lem:v-in-ball} 
Under the settings of Theorem \ref{thm:convergence}, fix a radius $r>0$. 
For some constant $\delta_r=\delta_r(d, \lambda, \sigma, \tempconst, \mathcal{D}) >0$ it holds
$$ \mathbb{P} \left(\overline{V}_k \in  B^\infty_r\left(\lambda (x^* - \overline{X}_k)\right) \right) \geq \rev{\delta_r}, \quad \text{ for all } \;\;k \geq 1\,.$$  
\end{lemma}
\begin{proof} 

From the update rule \eqref{eq:nonlinear-update}, we have 
\begin{multline} \label{eq:PV}
    \mathbb{P} \left(\overline{V}_{k+1} \in B^\infty_r\left(\lambda (x^* - \overline{X}_k)\right)\right) = (1-\Delta t) \mathbb{P}\left(\overline{V}_k \in B^\infty_r\left(\lambda (x^* - \overline{X}_k)\right)\right) \\ + \Delta t \mathbb{P} \left(\overline{W}_k \in B^\infty_r\left(\lambda (x^* - \overline{X}_k)\right)\right)\,.
\end{multline}
Now, consider an arbitrary coordinate $\ell \in \{1,\dots,d\}$. By definition of $\overline{W}_k$, see \eqref{eq:Wbar}, we note that the condition $(\overline{W}_k)_\ell \in B_r(\lambda(x^\ast - \overline{X}_k)_\ell)$ is equivalent to 
\[
\lambda(X^\alpha[\rho_k] - x^*)_\ell + \sigma(\sigma_0 + |(X^\alpha - \overline{X}_k)_\ell| )\overline{\xi}_{k,\ell} \in B_r(0)\,,
\]
which we rewrite for simplicity as $c_1 + B\overline{\xi}_{k,\ell} \in [-r,r]$ with $c_1 :=  \lambda(X^\alpha[\rho_k] - x^*)_\ell$, $c_2 := \sigma(\sigma_0 + |(X^\alpha - \overline{X}_k)_\ell| )$. Note that, thanks to the compactness assumption on $\DD$ and $\sigma,\sigma_0>0$, it holds $|c_1|\leq \lambda \textup{diam}(\DD)$ and $\sigma \sigma_0 \leq c_2 \leq \sigma(\sigma_0 + \textup{diam}(\DD))$. Since $\overline{\xi}_{k,\ell}\sim \mathcal{N}(0,1)$, we have $A + B\overline{\xi}_{k,\ell} \sim \mathcal{N}(c_1,c_2^2)$ and 
\begin{align*}
\mathbb{P}\left( (\overline{W}_k)_\ell \in B_r(\lambda(x^\ast - \overline{X}_k)_\ell) \right) & = \mathbb{P}\left(c_1 + c_2\overline{\xi}_{k,\ell} \in [-r,r] \right) \\
& = \frac1{\sqrt{2\pi c_2^2}}\int_{-r}^r e^{-\frac{(x-c_1)^2}{2c_2^2}}\, dx \\
& \geq \frac1{\sqrt{2\pi c_2^2}} e^{-\frac{c_1^2 + r^2}{2c_2^2}}
\geq \tilde{\delta}
\end{align*}
for some  $\tilde\delta=\tilde\delta(r, \lambda, \sigma, \tempconst, \mathcal{D})>0$, where we used the bounds on $|c_1|,c_2$. We note that for writing the density of $c_1 + c_2\xi_\ell$ the role of the additional parameter $\sigma_0$, which ensures that $c_2>0$, it is essential.

By iterating the argument for all $\ell = 1,\dots,d$ we obtain $\mathbb{P} \left(\overline{W}_k \in B^\infty_r\left(\lambda (x^* - \overline{X}_k)\right)\right) \geq \tilde{\delta}^d$ since $\overline{\xi}_{k,1},\dots,\overline{\xi}_{k,d}$ are independent. We plug this estimate in \eqref{eq:PV} and iterate for all time steps $k$ to obtain the desired estimate:
\begin{align*} \label{eq:PV}
    \mathbb{P} \left(\overline{V}_{k+1} \in B^\infty_r\left(\lambda (x^* - \overline{X}_k)\right)\right) &= (1-\Delta t) \mathbb{P}\left(\overline{V}_k \in B^\infty_r\left(\lambda (x^* - \overline{X}_k)\right)\right)+\dt\tilde{\delta}^d \\
    & \geq \min\left \{\mathbb{P}\left(\overline{V}_k \in B^\infty_r\left(\lambda (x^* - \overline{X}_k)\right)\right), \tilde{\delta}^d\right \} \\
    & \geq \min\left \{\mathbb{P}\left(\overline{V}_0 \in B^\infty_r\left(\lambda (x^* - \overline{X}_0)\right)\right), \tilde{\delta}^d\right \}=: \delta_r
\end{align*}
where $\delta_r=\delta_r(d, \lambda, \sigma, \tempconst, \mathcal{D},f_0)>0$ , since $\textup{supp}(f_0) = \DD\times \RR^d$ by assumption.
\end{proof}

Next, we use the above estimate in the velocity space, to provide an estimate on the mass around the minimizer.

\begin{proposition}  \label{p:mass}
Under the settings of Theorem \ref{thm:convergence}, fix a radius $r>0$.
 Then if $\lambda\Delta t\in [0,1]$, it holds
$$\rho_k(B_r(x^\ast)) \geq \delta_r^k\, \rho_0(B_r(x^\ast))  \,,  \quad \text{ for all }\; k\geq 0$$
where $\delta_r>0$ is the constant from Lemma \ref{lem:v-in-ball}. 
\end{proposition}
\begin{proof}
By Bayes' theorem and Lemma \ref{lem:v-in-ball}, we observe that
\begin{align*}
        \mathbb{P}(\overline{X}_{k+1} \in B^\infty_r(x^\ast)) &= \mathbb{P}(\overline{X}_k + \Delta t\overline{V}_{k+1} \in B^\infty_r(x^\ast)) \\
        &\geq \delta_r\, \mathbb{P} \left( \overline{X}_k + \Delta t\overline{V}_{k+1} \in B^\infty_r(x^\ast))\, \middle|\, \overline{V}_{k+1} \in B_r^\infty\left(\lambda(x^\ast-\overline{X}_k \right)\right)\,.
\end{align*}
Let $\Theta$ be the random variable such that $V_{k+1} = \lambda(\Theta - \overline{X}_k)$. Then, we have
\begin{align*}
\mathbb{P} \Big( \overline{X}_k + \Delta t\overline{V}_{k+1} \in B^\infty_r(x^\ast))\, \Big |&\, \overline{V}_{k+1} \in B_r^\infty\left(\lambda(x^\ast-\overline{X}_k \right)\Big)   \\
&  = \mathbb{P} \left( \overline{X}_k + \Delta t\lambda (\Theta - \overline{X}_k)) \in B^\infty_r(x^\ast))\, \middle|\, \Theta \in B_r^\infty\left(x^\ast\right) \right) \\
&  = \mathbb{P} \left( (1 - \dt \lambda) \overline{X}_k + \Delta t\lambda \Theta  \in B^\infty_r(x^\ast))\, \middle|\, \Theta \in B_r^\infty\left(x^\ast\right) \right) \\
& \geq \mathbb{P} \left(\overline{X}_k \in B^\infty_r(x^\ast)\right)
\end{align*}
where in the last inequality we used the convexity of $B^\infty_r(x^*)$ and $\lambda\Delta t \in [0,1]$. By the previous estimate, we obtain
\[
\mathbb{P}(\overline{X}_{k+1} \in B^\infty_r(x^\ast)) \geq \delta_r \,\mathbb{P}(\overline{X}_{k} \in B^\infty_r(x^\ast))\geq \delta_r^k\, \mathbb{P}(\overline{X}_{0} \in B^\infty_r(x^\ast))
\]
which is the desired lower bound.
\end{proof}

Next, we track the evolution of the error functional $\mathcal{H}[f_k]$ introduced in \eqref{eq:energy}.

\begin{proposition}\label{prop:evolution-of-errors} 
Under the settings of Theorem \ref{thm:convergence}, assume the parameters $\lambda, \sigma, \nu,\gamma >0$ are chosen such that
\begin{equation*}
\gamma \in \left[ \nu \frac{4\sigma \sqrt{d}}{\lambda} + 2\lambda  \,,\, \nu - 2\lambda \right]\,. 
\end{equation*} 
Then, for all time steps $k$ it holds
    \begin{equation}
        \mathcal{H}[f_{k+1}] \leq (1-A\dt)\mathcal{H}[f_{k}]+B\dt( \tempconst+|X^\alpha[\rho_k]-x^\ast| ),
    \end{equation}
    where $A, B$ are positive constants which depend on $\lambda, \sigma, \nu, \gamma$.
\end{proposition}
Before providing a proof, we note that there is a choice of $\lambda,\sigma,\nu$ for which the interval is well defined. For instance, if we take $\lambda > 4\sigma \sqrt{d}$, the set of admissible $\gamma$ is non-empty provided $\nu$ is sufficiently large with respect to $\lambda$. In terms of the algorithmic dynamics, this means that the particles should update their velocities sufficiently fast with respect to the position updates.

\begin{proof}
We first provide estimates for the step $k + \frac12$, corresponding to $f_{k+\frac12} = \textup{Law}(\overline{X}_k, \overline{V}_{k+1})$, and then for $f_{k+1}$. We introduce the following notation for simplicity:
\begin{equation*} 
\mathcal{X}_k = \BE[|\overline{X}_k - x^\ast|], \quad \mathcal{V}_k = \BE[|\overline{V}_k - \lambda(x^\ast - \overline{X}_k)|], \quad \textup{and}\quad \mathcal{V}_{k+\frac12} = \BE[|\overline{V}_{k+1} - \lambda(x^\ast - \overline{X}_k)|],
\end{equation*}
so that $\mathcal{H}[f_k] = \gamma\mathcal{X}_k + \mathcal{V}_k$. 
At step $k  + \frac12$, from the update rule \eqref{eq:nonlinear-update}, it holds
\begin{align*}
 \mathcal{V}_{k+\frac{1}{2}}&=\BE[|\overline{V}_{k+1}-\lambda(x^\ast-\overline{X}_k)|] \\
    &=(1- e^{-\nu \Delta t})\mathcal{V}_k+ e^{-\nu\dt} \BE \left[|\overline{W}_k -\lambda(x^\ast-\overline{X}_k)|\right]
    %&=(1-e^{-\nu\dt})\mathcal{V}_k+ e^{-\nu\dt} \BE \left[|\lambda (X^\alpha[\rho_k] -x^\ast)+\sigma \sum_{\ell=1}^d(\tempconst+|(X^\alpha[\rho_k]-x)_\ell|)\overline{\xi}_{k,\ell} e_\ell|\right].
\end{align*}
Let $\{e_\ell\}_{\ell=1}^d$ be the standard base for $\RR^d$. By definition of $\overline{W}_k$, and using the triangular inequality, the second term can be estimated as
\begin{align*}
\BE& \left[|\overline{W}_k -\lambda(x^\ast-\overline{X}_k)|\right]  \\
& =\BE \left[|\lambda (X^\alpha[\rho_k] -x^\ast)+\sigma \sum_{\ell=1}^d(\tempconst+|(X^\alpha[\rho_k]-x)_\ell|)\overline{\xi}_{k,\ell} e_\ell|\right] \\
& \leq \BE\left[|(\lambda + \sigma\odot \overline{\xi}_k) (X^\alpha[\rho_k] - x^*) |\right]
 + \sigma\sum_{\ell=1}^d (\sigma_0 + \BE|(\overline{X}_k - x^*)_\ell| ) \BE[|\overline{\xi}_{k,\ell}|] \\
 & \leq (\lambda + \sigma \sqrt{d})|X^\alpha[\rho_k] - x^*| + \sigma \sigma_0 \sqrt{d} + \sigma\sqrt{d}\mathbb{E}[|\overline{X}_k - x^* |]
\end{align*}
where we used that $\mathbb{E}[|\overline{\xi}_{k,\ell}|]  =1 $, and the norms inequalities \eqref{eq:norms}. 
To obtain estimates in terms of $\nu\dt$, we use:
$
1 - e^{\nu\dt} \leq \nu \dt 
$
and take $\dt>0$ sufficiently small such that $e^{-\nu\dt}\leq (1 - \nu\dt/2)$. Altogether, we obtain
\begin{equation}\label{eq:v+1/2}
    \mathcal{V}_{k+\frac{1}{2}} \leq (1-\nu\dt/2)\mathcal{V}_k+ \nu\dt \left( \sigma \sqrt{d}\mathcal{X}_k+\Xi_k\right)\,,
\end{equation}
where we write $$\Xi_k = (\lambda+\sigma\sqrt{d})|X^\alpha[\rho_k]-x^\ast|+\sigma \tempconst \sqrt{d}\,,$$ which, we note, we will be able to bound using the quantitative Laplace principle (Proposition \ref{prop:quant-laplace}). 
For step $k+1$, we derive the evolution in terms of $\mathcal{X}_k, \mathcal{V}_{k +\frac12}$ as follows:
\begin{align*}
    \mathcal{X}_{k+1}&=\BE[|\overline{X}_{k+1}-x^\ast|]=\BE[|\overline{X}_k+ \dt\overline{V}_{k+1}-x^\ast|] \\
    &=\BE [|\overline{X}_k-x^\ast+\dt(\overline{V}_{k+1}-\lambda(x^\ast-\overline{X}_k))+\lambda \dt(x^\ast-\overline{X}_k)|] \\
    &\leq (1-\lambda \dt)\BE[|\overline{X}_k-x^\ast|] + \dt \BE[|\overline{V}_{k+1}-\lambda(x^\ast-\overline{X}_k)|] \\
    &= (1-\lambda \dt)\mathcal{X}_k+\dt\mathcal{V}_{k+\frac{1}{2}}
\end{align*}
where we simply used triangular inequality. For $\mathcal{V}_{k+1}$ it holds
\begin{align*}
    \mathcal{V}_{k+1} &= \BE[|\overline{V}_{k+1}-\lambda(x^\ast-\overline{X}_{k+1})|] = \BE [|\overline{V}_{k+1}-\lambda(x^\ast-\overline{X}_{k}- \overline{V}_{k+1}\dt)|] \\
    &\leq \BE[|\overline{V}_{k+1}-\lambda(x^\ast-\overline{X}_{k})|]+\BE [|\lambda \dt (\overline{V}_{k+1}-\lambda(x^\ast-\overline{X}_{k}))+\lambda^2\dt(x^\ast-\overline{X}_{k})|] \\
    &\leq (1+\lambda\dt)\BE[|\overline{V}_{k+1}-\lambda(x^\ast-\overline{X}_{k})|]+\lambda^2\dt \BE[|\overline{X}_{k}-x^\ast|] \\
    &= (1+\lambda \dt)\mathcal{V}_{k+\frac{1}{2}}+\lambda^2\dt\mathcal{X}_k\,.
\end{align*}
Collecting all the results with plugging in the inequality \eqref{eq:v+1/2}, we get 
\begin{align*}
    \gamma \mathcal{X}_{k+1}+\mathcal{V}_{k+1} & \leq \gamma(1-\lambda \dt)\mathcal{X}_k+\gamma\dt\mathcal{V}_{k+\frac{1}{2}}+(1+\lambda\dt)\mathcal{V}_{k+\frac{1}{2}}+\lambda^2 \dt \mathcal{X}_k \\
& \leq \gamma(1-\lambda \dt + \lambda^2 \dt )\mathcal{X}_k + (1+\lambda\dt + \gamma\dt)\mathcal{V}_{k+\frac{1}{2}}\\
& \leq \gamma(1-\lambda \dt + \lambda^2 \dt)\mathcal{X}_k + (1+\lambda\dt + \gamma\dt)\left((1-\nu\dt/2)\mathcal{V}_k+ \nu\dt \left( \sigma \sqrt{d}\mathcal{X}_k+\Xi_k\right)\right) \\
& =: C_1 \gamma \mathcal{X}_k  + C_2 \mathcal{V}_k + C_3 \Xi_k
\end{align*}
with 
\begin{align*}
C_1 &= \frac1\gamma\left(\gamma(1-\lambda \dt)+\lambda^2 \dt+\nu \sigma \sqrt{d}\dt (1+(\gamma+\lambda)\dt)\right) \\
C_2 &=  (1-\nu \dt/2)(1+(\gamma+\lambda)\dt)   \\
C_3 &= \nu \dt(1+(\gamma+\lambda)\dt)\,. 
\end{align*}
Taking $\dt$ such that $\dt(\gamma + \lambda)\leq 1$, we have
\begin{align*}
C_1 & \leq (1 - \lambda \dt) + \frac{\lambda^2}\gamma\dt + \frac{\nu \sigma \sqrt{d}}\gamma = \left( 1 - \left(\lambda - \lambda^2/\gamma - \nu\sigma \sqrt{d}/\gamma \right)\dt \right ) \\
&=: (1 - A_1\dt)\\
C_2 &  = 1 - \nu \dt/2 + (\gamma + \lambda)\dt -\nu(\gamma+\lambda)\dt^2/2 \leq \left(1 - (\nu/2 -\gamma - \lambda)\dt \right) \\
&=: (1 - A_2\dt)\\
C_3 &= \nu \dt(1+(\gamma+\lambda)\dt) \leq 2\nu\dt\,.
\end{align*}
Now, to obtain a decay rate for $\mathcal{H}[f_k]$, we need to select the parameters $\lambda, \sigma, \nu, \gamma>0$ such that $A_1, A_2>0$. In terms of $\gamma$, this translated into the condition
\begin{equation*}
\gamma \in \left[ \nu \frac{4\sigma \sqrt{d}}{\lambda} + 2\lambda  \,,\, \nu - 2\lambda \right]\,. 
\end{equation*}
 Once the set of parameters $\lambda, \sigma, \nu, \gamma$ is picked, the desired estimate is obtained by setting $A := \min\{A_1, A_2\}$, $B = 2\nu(\lambda + \sigma \sqrt{d})$.
\end{proof}

We are finally ready to prove the main convergence theorem. 

\begin{proof}[Proof of Theorem \ref{thm:convergence}]
We will prove the statement by contradiction. Let us assume that $\mathcal{H}[f_k]>\e$  for all $k\dt \leq T^\star$ and recall that the error evolution from Proposition~\ref{prop:evolution-of-errors} was given as 
\begin{equation*}
    \mathcal{H}[f_k] \leq (1-A\dt)\mathcal{H}[f_{k-1}]+B\dt( \tempconst+|X^\alpha[\rho_{k-1}]-x^\ast| ).
\end{equation*}
for a suitable choice of $\lambda,\nu,\sigma,\dt$. We will also assume $\sigma_0 \leq \e A/(2B)$.

In order to bound the term $|X^\alpha[\rho_{k-1}]-x^\ast| $,  we apply the quantitative Laplace principle (Proposition~\ref{prop:quant-laplace}) for all $k\dt \in [0, T^\star]$ with $q_\e, r_\e$ given by 
\begin{equation*}
    q_\e:=\frac{1}{2}\min \left\{ \left(\frac{\e A}{4c_p B\sqrt{d}} \right)^p, \EE_\infty \right\},\quad  r_\e:=\max_{s \in [0,R_p]} \left\{ \EE_s=\sup_{x \in B_s(x^\ast)} \EE(x) \leq q_\e\right\}.
\end{equation*}
Then, thanks to A2) from Assumption~\ref{assume:convergence} and the choice of $q_\e$, we get 
\begin{align*}
    \vert X^\alpha[\rho_{k-1}]-x^\ast \vert &\leq c_p\sqrt{d}(q_\e+\EE_{r_\e})^{1 \over p} + \frac{\sqrt{d}e^{-\alpha q_\e}}{\rho_{k-1}(B_{r_\e}(x^\ast))} \int |x^\ast-x|d \rho_{k-1}(x) \\
    &\leq \frac{\e A}{4B}+\frac{\sqrt{d}e^{-\alpha q_\e}}{\gamma\delta^k_{r_\e}} \mathcal{H}[f_{k-1}].
\end{align*}
Note that we used the estimate around the minimizer from Proposition~\ref{p:mass} with $r = r_\e$,  and $\int |x^\ast-x|d \rho_{k-1}(x) \leq \mathcal{H}[f_{k-1}]/\gamma$ for the second term. Now, we choose $\alpha=\alpha_\e$ sufficiently large so that it satisfies 
\begin{equation*}
    \frac{\sqrt{d}e^{-\alpha_\e q_\e}}{\gamma\,\delta^{T^\star/\dt}_{r_\e}} \leq \frac{A}{4B},
\end{equation*}
which then gives us 
\begin{equation*}
     \vert X^\alpha[\rho_{k-1}]-x^\ast \vert \leq  \frac{A}{2B}\mathcal{H}[f_{k-1}].
\end{equation*}
We note that there was no assumption on the parameter $\alpha$ in the previous auxiliary results.

Plugging in this results to the error evolution, it yields
\begin{align*}
    \mathcal{H}[f_k] &\leq (1-A\dt)\mathcal{H}[f_{k-1}]+B\dt \left(\tempconst+\frac{A} {2B}\mathcal{H}[f_{k-1}] \right) \\
    &\leq  \left(1- A\dt /2 \right) \mathcal{H}[f_{k-1}] + \tempconst B\dt.
\end{align*}
By iterating the above estimation, we get 
\begin{align*}
    \mathcal{H}[f_k] &\leq (1-A\dt/2)^k\mathcal{H}[f_0]+\tempconst B \dt \sum_{h=0}^{k-1} (1-A\dt/2)^h \\
    &\leq e^{-kA\dt/2} \mathcal{H}[f_0]+ \tempconst B \dt \frac{1-(1-A\dt/2)^k}{A \dt/2} \\
    &\leq e^{-kA\dt/2} \mathcal{H}[f_0]+ \frac{2\tempconst B}{A}(1-(1-A\dt/2)^k) \\
    &\leq  e^{-kA\dt/2} \mathcal{H}[f_0] + \frac{\e}{2},
\end{align*}
where we used our choice of $\tempconst$ in the last equality. 
Finally, thanks to the choice of $T^\star$,  for $k$ sufficiently large it holds
\begin{equation*}
    \mathcal{H}[f_k] \leq  \frac{\e}{2}+\frac{\e}{2}=\e,
\end{equation*}
which contradicts our initial assumption $\mathcal{H}[f_k] > \e$ for all $k\dt \leq T^\star$. 

\end{proof}

%%%%%%%%%%%%%%%%%%%%%%%%%%%%%%%%%%%%%%%%%%%%%%%%%%%%%%%%%%%%%%%%%%%%%%%%%%%%
\section{Numerical experiments}
\label{sec:numerics}

In this section, we present several numerical experiments to evaluate the performance of the proposed algorithm~\eqref{eq:psojump}. First, we assess its effectiveness on benchmark functions, focusing on identifying suitable parameter ranges for the diffusion parameter $\sigma$ as well as comparison on two aforementioned noise types, Gaussian and Cauchy. Next, we compare \rev{the proposed algorithm with PSO and, in particular,} the original CBO method by analysing performance as the diffusion scaling parameter $\e$ approaches zero.

In the implementation, the particle positions are rescaled to lie within the domain $[-1,1]^d$, and initial positions are sampled uniformly from this domain. At each iteration, if particles move outside the domain, boundary conditions are applied to enforce confinement within $[-1,1]^d$. We implemented a stopping criterion based on methodologies from \cite{benfenatiBinaryInteractionMethods2022, borghiConsensusBasedOptimization2023}. Specifically, we maintain a counter $n$ that tracks consecutive occurrences where $|X^\alpha_{\text{prev}}-X^\alpha_{\text{current}}|$ falls below a given tolerance threshold $\delta_{\text{stall}}>0$. When this condition persists for more than $n_{\text{stall}}$ consecutive iterations, we assume that the particle system has identified a solution and stop the computation. In the following experiments, these parameters are set to $\delta_{\text{stall}}=10^{-4}, n_{\text{stall}}=500$.

\subsection{Tests on benchmark problems}

We implemented the proposed \rev{PSO with jumps} algorithm \eqref{eq:psojump} and tested it on benchmark functions for optimization. The functions we used can be found in \cite{jamilLiteratureSurveyBenchmark2013} and are summarised in Table~\ref{table:test_ftns}. For all test functions, the location of the minimizer is known, and we set the search space dimension to $d = 20$. All of these functions are continuous but not necessarily differentiable or convex.

To assess the performance of our algorithm, we employ metrics such as success rate and the mean values of fitness values $\EE(X^\alpha_T)$ and  $\ell^\infty$ errors, i.e. $\lVert X^\alpha_T -x^\ast \rVert_\infty$, where $X^\alpha_T$ denotes the consensus point at the final iteration step. We note that $\EE(X^\alpha_T)$ also corresponds to the optimality gap since $\EE(x^*) = 0$ for all test functions.
Following \cite{pinnauConsensusbasedModelGlobal2017}, we consider an optimization successful if $X^\alpha_T$ lies within a $\ell^\infty$-ball with radius $0.25$ centered at the minimizer $x^\ast$. All results presented in the following sections were obtained from $100$ independent realisations of the algorithm.

\renewcommand{\arraystretch}{1.5}
\begin{table}
    \centering
    \small
    \begin{tabular}{|c|c|c|c|}
        \hline
        Name &  Function $\mathcal{F}(x)$  & Range & $x^*$ \\ \cline{1-4}
        
        Ackley & $-20\text{exp} \left( -\frac{0.2}{\sqrt{d}} \|x\| \right) - \text{exp} \left( -\frac{1}{d} \sum_{i=1}^d \cos{2\pi x_i}\right)  +20 +e $ & $[-5, 5]^d$ & $(0, \ldots, 0)$  \\  \cline{1-4}

        Rastrigin & $ \frac{1}{d} \sum_{i=1}^d [ x_i^2 -10\cos(2 \pi x_i) +10] $ & $[-5.12, 5.12]^d$ & $(0, \ldots, 0)$ \\  \cline{1-4}

        Griewank & $ 1+ \sum_{i=d}^d \frac{(x_i)^2}{4000} - \prod_{i=1}^d \cos{\frac{x_i}{i}}$ & $[-600, 600]^d$ & $(0, \ldots, 0)$ \\  \cline{1-4}

        Rosenbrock & $ \sum_{i=1}^{d-1} [100(x_{i+1}-x_i^2)^2 + (1-x_i)^2] $ & $[-100, 100]^d$ & $(1, \ldots, 1)$  \\  \cline{1-4}

        Salomon & $ 1-\cos \left( 2\pi \sqrt{\sum_{i=1}^d (x_i)^2}\right) +0.1\sqrt{\sum_{i=1}^d (x_i)^2} $ & $[-100, 100]^d$ & $(0, \ldots, 0)$ \\  \cline{1-4}

        Schwefel 2.20 & $ \sum_{i=1}^d |x_i| $ & $[-100, 100]^d$ & $(0, \ldots, 0)$ \\ \cline{1-4}

        XSY random & $ \sum_{i=1}^d \eta_i|x_i|^i $, $\eta_i \sim \mathcal{U}(0,1)$ & $[-5, 5]^d$ & $(0, \ldots, 0)$ \\ \cline{1-4}
    \end{tabular}
    \caption{Test functions for global non-convex optimization}
    \label{table:test_ftns}
\end{table}

As in PSO and CBO algorithms, the strength of the random component, in this case the diffusion coefficient $\sigma$, is regarded as the most influential parameter among all tunable parameters, as it governs the exploration of the search space. Consequently,  we investigate the optimal range of $\sigma$ for algorithm performance. We expect small values of $\sigma$ to lead to premature convergence, while large values may prevent the emergence of consensus at all.
As suggested by the theoretical analysis (see Theorem \ref{thm:convergence}), we also expect large values of the weight $\alpha$ to yield better results. Therefore, we fix $\alpha=10^5$ throughout our experiments. Setting $\alpha$ to such a large value may lead to numerical instability in the evaluation of the weight function $\omega_\alpha(x)=e^{-\alpha \EE(x)}$ due to potential overflow or underflow in the exponential calculation. To prevent this, we used the numerical trick introduced in \cite{fornasierConsensusbasedOptimizationSphere2021}. 

\begin{figure}
    \centering
    \begin{subfigure}{0.48\textwidth}
        \centering
        \includegraphics[width=\textwidth]{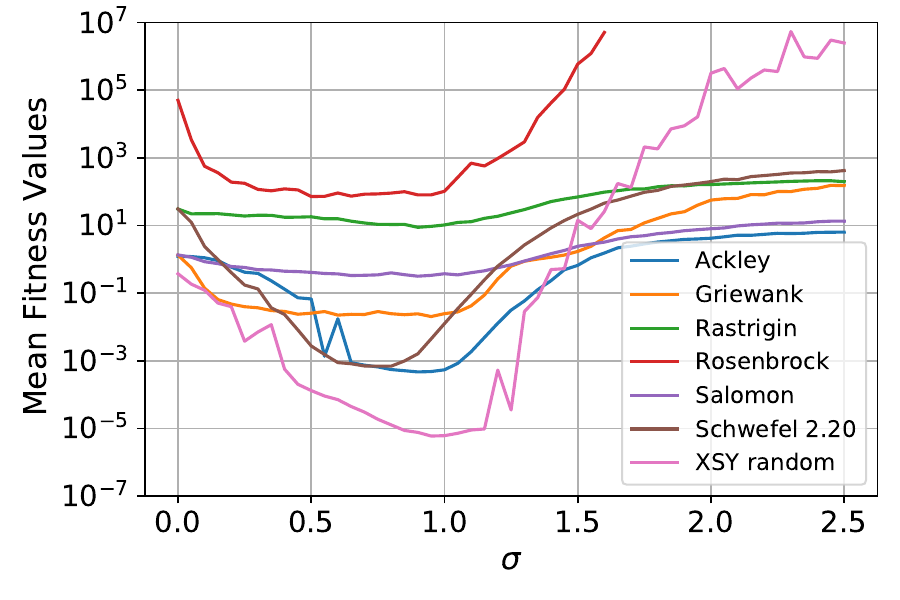}
        \caption{$\BE[\EE(X^T_\alpha)]-\inf_x\EE(x)$, Gaussian noise}
        \label{fig:gauss}
    \end{subfigure}
    \hfill    
    \begin{subfigure}{0.48\textwidth}
        \centering
        \includegraphics[width=\textwidth]{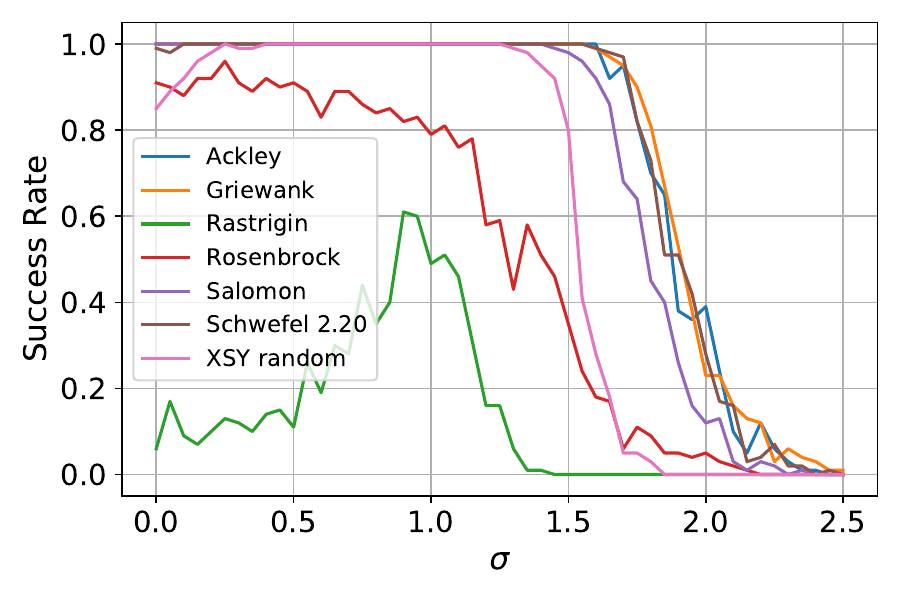}
        \caption{Success rate, Gaussian noise}
        \label{fig:gauss_success}
    \end{subfigure}
    
    \begin{subfigure}{0.48\textwidth}
        \centering
        \includegraphics[width=\textwidth]{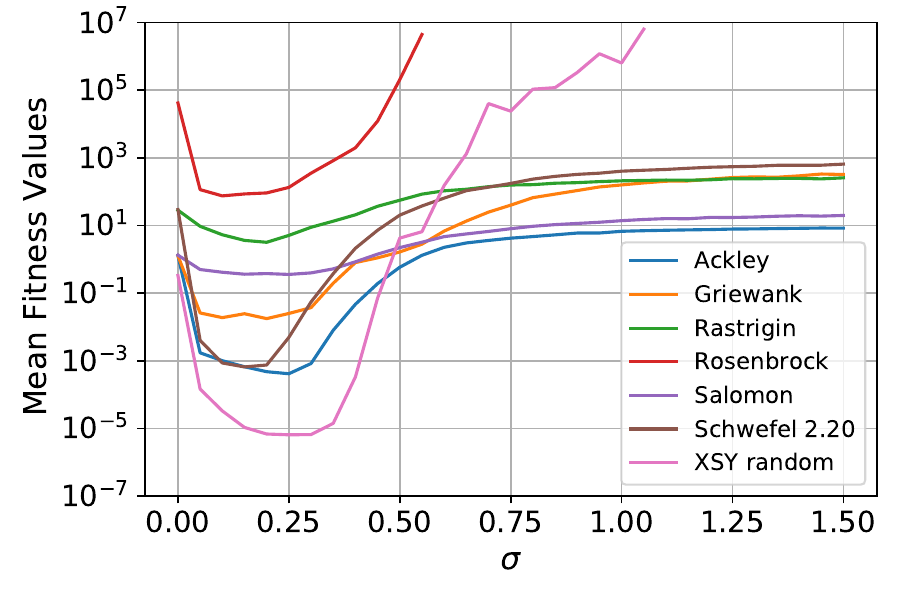}
        \caption{$\BE[\EE(X^T_\alpha)]-\inf_x \EE(x)$, Cauchy noise}
        \label{fig:cauchy}
    \end{subfigure}  
    \hfill
        \begin{subfigure}{0.48\textwidth}
        \centering
        \includegraphics[width=\textwidth]{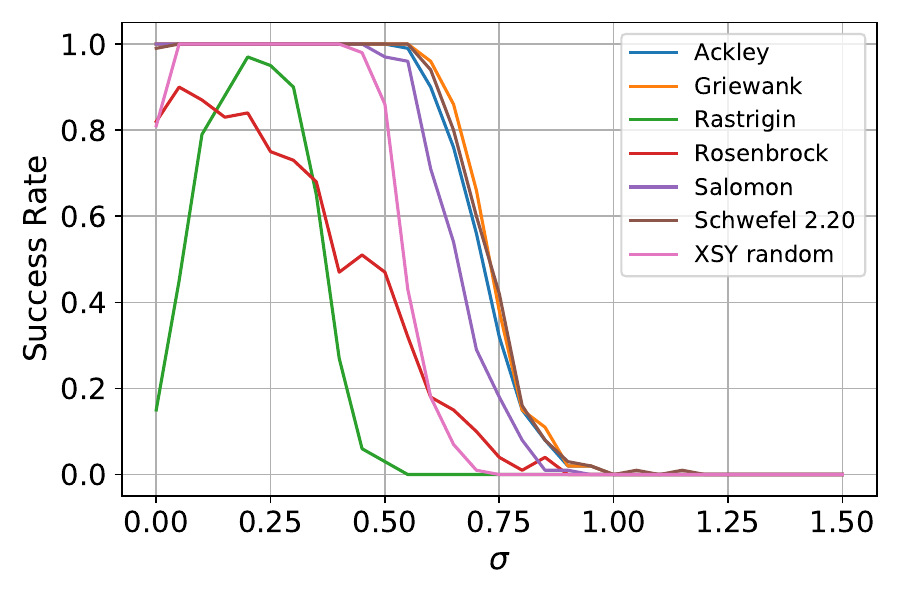}
        \caption{Success rate, Cauchy noise}
        \label{fig:cauchy_success}
    \end{subfigure} 
    \caption{Variations in diffusion coefficient $\sigma$ using Gaussian (top) and Cauchy (bottom) noise.  Parameters are set to $N=200, \dt=0.1, \lambda=1, \nu=1, \alpha=10^5$, with total iteration step $k_T=10^3$. For visualisation purposes, values exceeding $10^7$ were excluded. We note that, although the mean $l^\infty$ error is small, the fitness values can still become large due to the definition of the Rosenbrock and XSY random function.}
    \label{fig:sigma}
\end{figure}

The top plots in Figure~\ref{fig:sigma} illustrate the performance of the algorithm across a range of diffusion coefficients, from $\sigma = 0$ to $\sigma = 2.5$, using Gaussian noise. While the optimal choice of $\sigma$ varies across test functions, a common range of $\sigma \in [0.75, 1]$ tends to yield consistently good results for most functions, with the exceptions of the Rastrigin and Rosenbrock functions.

The bottom plots present analogous results using standard Cauchy noise for $\sigma \in [0, 1.5]$. Due to the heavy-tailed nature of the Cauchy distribution, better performance is observed for smaller values of $\sigma$, as reflected in Figure~\ref{fig:sigma}. In particular, the Cauchy noise proves highly effective for the Rastrigin function, achieving a success rate close to $100\%$ with $\sigma = 0.2$. Based on these findings, we identify $\sigma = 0.75$ as a generally effective choice for Gaussian noise, and $\sigma = 0.25$ for Cauchy noise.
 
\begin{table}[t]
\centering
%\scriptsize  % optional: shrink font to fit better
\resizebox{\textwidth}{!}{
\begin{tabular}{c|c|cc|cc|cc|cc|cc}
\toprule
\multirow{2}{*}{Name} & \multirow{2}{*}{} 
& \multicolumn{2}{c|}{$N=50$} 
& \multicolumn{2}{c|}{$N=100$} 
& \multicolumn{2}{c|}{$N=200$} 
& \multicolumn{2}{c|}{$N=500$} 
& \multicolumn{2}{c}{$N=1000$} \\
\cmidrule(lr){3-4} \cmidrule(lr){5-6} \cmidrule(lr){7-8} \cmidrule(lr){9-10} \cmidrule(lr){11-12}
& & Gaussian & Cauchy & Gaussian & Cauchy & Gaussian & Cauchy & Gaussian & Cauchy & Gaussian & Cauchy \\
\midrule
\multirow{3}{*}{Ackley} 
& Success Rate & 97\% & 100\% & 100\% & 100\% & 100\% & 100\% & 100\% & 100\% & 100\% & 100\% \\
& $\BE[\lVert X^\alpha_T-x^\ast \rVert_\infty]$        
& 6.33E-02 & 2.69E-03 
& 5.43E-03 & 1.78E-04
& 6.16E-05 & 3.97E-05 
& 4.89E-05 & 2.10E-05 
& 4.42E-05 & 1.95E-05 \\
& $\BE[\EE(X^\alpha_T)]$ 
& 1.56E-01 & 2.38E-02 
& 9.83E-03 & 1.56E-03 
& 1.30E-04 & 4.14E-04 
& 1.14E-04 & 2.56E-04 
& 1.09E-04 & 2.54E-04 \\
\midrule
\multirow{3}{*}{Rastrigin} 
& Success Rate & 5\% & 25\% & 23\% & 75\% & 54\% & 100\% & 85\% & 100\% & 99\% & 100\% \\
& $\BE[\lVert X^\alpha_T-x^\ast \rVert_\infty]$        & 4.87E-01 & 3.91E-01
& 3.75E-01 & 2.49E-01
& 2.86E-01 & 1.94E-01
& 2.20E-01 & 9.78E-02
& 1.89E-01 & 5.99E-03 \\
& $\BE[\EE(X^\alpha_T)]$ 
& 1.11E+02 & 1.79E+01
& 7.84E+01 & 1.01E+01
& 5.20E+01 & 5.05E+00
& 2.99E+01 & 8.10E-01
& 1.64E+01 & 3.34E-02 \\
\midrule
\multirow{3}{*}{Griewank} 
& Success Rate & 100\% & 100\% & 100\% & 100\% & 100\% & 100\% & 100\% & 100\% & 100\% & 100\% \\
& $\BE[\lVert X^\alpha_T-x^\ast \rVert_\infty]$        & 1.37E-02 & 1.25E-02 & 1.12E-02 & 1.20E-02 & 9.36E-03 & 1.07E-02 & 1.03E-02 & 9.34E-03 & 9.85E-03 & 9.25E-03 \\
& $\BE[\EE(X^\alpha_T)]$ & 4.46E-05 & 4.07E-05 & 3.36E-05 & 3.39E-05 & 2.68E-05 & 2.93E-05 & 3.04E-05 & 2.68E-05 & 2.65E-05 & 2.57E-05 \\
\midrule
\multirow{3}{*}{Rosenbrock} 
& Success Rate & 64\% & 53\%& 83\% & 65\%& 84\% & 75\%& 93\% & 83\%& 90\% & 84\%\\
& $\BE[\lVert X^\alpha_T-x^\ast \rVert_\infty]$        & 2.12E-01 & 3.37E-01
& 1.09E-01 & 2.26E-01
& 8.81E-02 & 1.55E-01
& 6.08E-02 & 1.28E-01
& 6.82E-02 & 1.26E-01
\\
& $\BE[\EE(X^\alpha_T)]$ & 1.02E+03 & 1.83E+03
& 1.61E+02 &6.05E+02
 & 8.68E+01 & 1.33E+02
& 3.65E+01 &5.52E+01
 & 2.85E+01 & 4.21E+01
\\
\midrule
\multirow{3}{*}{Salomon} 
& Success Rate & 100\% & 100\% & 100\% & 99\% & 100\% & 100\% & 100\% & 100\% & 100\% & 100\% \\
& $\BE[\lVert X^\alpha_T-x^\ast \rVert_\infty]$        & 3.32E-02 & 3.82E-02 & 1.91E-02 & 2.11E-02 & 1.45E-02 & 1.46E-02 & 1.02E-02 & 9.98E-03 & 8.96E-03 & 7.95E-03 \\
& $\BE[\EE(X^\alpha_T)]$ & 6.48E-01 & 7.96E-01 & 4.39E-01 & 4.56E-01 & 3.56E-01 & 3.39E-01 & 3.87E-01 & 3.39E-01 & 2.71E-01 & 4.83E-01 \\
\midrule
\multirow{3}{*}{Schwefel 2.20} 
& Success Rate & 100\% & 100\% & 100\% & 99\% & 100\% & 100\% & 100\% & 100\% & 100\% & 100\% \\
& $\BE[\lVert X^\alpha_T-x^\ast \rVert_\infty]$        & 4.05E-03 & 1.99E-03 & 3.20E-05 & 1.59E-04 & 1.53E-06 & 1.01E-05 & 7.73E-07 & 5.87E-07 & 6.53E-07 & 3.54E-07 \\
& $\BE[\EE(X^\alpha_T)]$ & 5.44E-01 & 7.99E-01 & 7.29E-03 & 7.24E-02 & 6.78E-04 & 4.96E-03 & 4.81E-04 & 4.12E-04 & 4.48E-04 & 2.94E-04 \\
\midrule
\multirow{3}{*}{XSY random} 
& Success Rate & 98\% & 100\% & 100\% & 99\% & 100\% & 100\% & 100\% & 100\% & 100\% & 100\% \\
& $\BE[\lVert X^\alpha_T-x^\ast \rVert_\infty]$        & 1.09E-01 & 6.54E-02 & 9.06E-02 & 4.34E-02 & 8.81E-02 & 3.48E-02 & 8.20E-02 & 3.08E-02 & 7.82E-02 & 3.03E-02 \\
& $\BE[\EE(X^\alpha_T)]$ & 4.95E-02 & 5.11E-05 & 3.39E-05 & 8.30E-06 & 2.00E-05 & 6.47E-06 & 1.31E-05 & 3.89E-06 & 7.83E-06 & 2.96E-06 \\

\bottomrule
\end{tabular}}

\caption{Variations in particle number $N$. Parameters are set to $\sigma=0.75 \text{ (Gaussian)}, \sigma=0.25 \text{ (Cauchy)}, \dt=0.1, \lambda=1, \nu=1, \alpha=10^5$, with total iteration step $k_T=10^3$.}
\label{table:test}
\end{table}

Table~\ref{table:test} summarises the performance of our algorithm on all of the functions in Table~\ref{table:test_ftns} for varying particle numbers ($N = 50, 100, 200, 500, 1000$), using both Gaussian and Cauchy noise. Based on the results presented in Figure~\ref{fig:sigma}, we selected $\sigma = 0.75$ for Gaussian noise and $\sigma = 0.25$ for Cauchy noise. The algorithm performs well across most of the functions, showing a clear trend of decreasing error as the number of particles increases. For the Rastrigin and Rosenbrock functions, however, the chosen values of $\sigma$ may not be optimal, which could partly explain their comparatively lower performance. In particular, the Rastrigin function proved more challenging to optimise. This difficulty is well-documented in the CBO literature~\cite{benfenatiBinaryInteractionMethods2022, borghiConsensusBasedOptimization2023, grassiParticleSwarmOptimization2021}, and can be attributed to the function’s structural characteristics; most notably, the presence of a global minimum surrounded by a dense array of local minima.

\begin{table}[t]
\centering
%\scriptsize  % optional: shrink font to fit better
\resizebox{\textwidth}{!}{
\begin{tabular}{c|c|cc|cc|cc|cc|cc}
\toprule
\multirow{2}{*}{Name} & \multirow{2}{*}{} 
& \multicolumn{2}{c|}{$N=50$} 
& \multicolumn{2}{c|}{$N=100$} 
& \multicolumn{2}{c|}{$N=200$} 
& \multicolumn{2}{c|}{$N=500$} 
& \multicolumn{2}{c}{$N=1000$} \\
\cmidrule(lr){3-4} \cmidrule(lr){5-6} \cmidrule(lr){7-8} \cmidrule(lr){9-10} \cmidrule(lr){11-12}
& & BGK & PSO & BGK & PSO & BGK & PSO & BGK & PSO & BGK & PSO \\
\midrule
\multirow{3}{*}{Ackley} 
& Success Rate & 100\% & 100\% & 100\% & 100\% & 100\% & 100\% & 100\% & 100\% & 100\% & 100\% \\
& $\BE[\lVert X^\alpha_T-x^\ast \rVert_\infty]$  
& 2.69E-03 & 9.18E-06
& 1.78E-04 & 5.94E-06
& 3.97E-05 & 3.85E-06 
& 2.10E-05 & 2.94E-06 
& 2.10E-05 & 2.52E-06 \\
& $\BE[\EE(X^\alpha_T)]$ 
& 2.38E-02 & 9.71E-05 
& 1.56E-03 & 7.20E-05 
& 4.14E-04 & 5.42E-05
& 2.56E-04 & 4.42E-05
& 2.54E-04 & 4.00E-05 \\
\midrule
\multirow{3}{*}{Rastrigin} 
& Success Rate & 25\% & 20\% & 75\% & 51\% & 100\% & 68\% & 100\% & 92\% & 100\% & 100\% \\
& $\BE[\lVert X^\alpha_T-x^\ast \rVert_\infty]$        
& 3.91E-01 & 3.74E-01
& 2.49E-01 & 2.92E-01
& 1.94E-01 & 2.57E-01
& 9.78E-02 & 2.10E-01
& 5.99E-03 & 1.89E-01 \\
& $\BE[\EE(X^\alpha_T)]$ 
& 1.79E+01 & 1.46E+01
& 1.01E+01 & 1.04E+01
& 5.05E+00 & 7.50E+00
& 8.10E-01 & 4.48E+00
& 3.34E-02 & 3.27E+00
\\

\bottomrule
\end{tabular}}

% \begin{tabular}{c|c|cc|cc|cc|cc|cc}
% \toprule
% \multirow{2}{*}{Name} & \multirow{2}{*}{} 
% & \multicolumn{2}{c|}{$N=50$} 
% & \multicolumn{2}{c|}{$N=100$} 
% & \multicolumn{2}{c|}{$N=200$} 
% & \multicolumn{2}{c|}{$N=500$} 
% & \multicolumn{2}{c}{$N=1000$} \\
% \cmidrule(lr){3-4} \cmidrule(lr){5-6} \cmidrule(lr){7-8} \cmidrule(lr){9-10} \cmidrule(lr){11-12}
% & & BGK & PSO & BGK & PSO & BGK & PSO & BGK & PSO & BGK & PSO \\
% \midrule
% \multirow{3}{*}{Ackley} 
% & Success Rate & 97\% & 100\% & 100\% & 100\% & 100\% & 100\% & 100\% & 100\% & 100\% & 100\% \\
% & $\BE[\lVert X^\alpha_T-x^\ast \rVert_\infty]$  
% & 6.33E-02 & 9.18E-06
% & 5.43E-03 & 5.94E-06
% & 6.16E-05 & 3.85E-06 
% & 4.89E-05 & 2.94E-06 
% & 4.42E-05 & 2.52E-06 \\
% & $\BE[\EE(X^\alpha_T)]$ 
% & 1.56E-01 & 9.71E-05 
% & 9.83E-03 & 7.20E-05 
% & 1.30E-04 & 5.42E-05
% & 1.14E-04 & 4.42E-05
% & 1.09E-04 & 4.00E-05 \\
% \midrule
% \multirow{3}{*}{Rastrigin} 
% & Success Rate & 5\% & 20\% & 23\% & 51\% & 54\% & 68\% & 85\% & 92\% & 99\% & 100\% \\
% & $\BE[\lVert X^\alpha_T-x^\ast \rVert_\infty]$        
% & 4.87E-01 & 3.74E-01
% & 3.75E-01 & 2.92E-01
% & 2.86E-01 & 2.57E-01
% & 2.20E-01 & 2.10E-01
% & 1.89E-01 & 1.89E-01 \\
% & $\BE[\EE(X^\alpha_T)]$ 
% & 1.11E+02 & 1.46E+01
% & 7.84E+01 & 1.04E+01
% & 5.20E+01 & 7.50E+00
% & 2.99E+01 & 4.48E+00
% & 1.64E+01 & 3.27E+00
% \\

% \bottomrule
% \end{tabular}}

\caption{\rev{Comparison between the proposed jump PSO algorithm \eqref{eq:psojump} (with Cauchy noise) and PSO \eqref{eq:sdpso}. Particle number $N$ is varied for both cases. Parameters are set to $m=0.1$, $\gamma = 1-m$, $\lambda_1 = \lambda_2 = 1$, $\sigma_1 = \sigma_2 = 1.7$, $\Delta t = 0.1$, $\alpha = 10^5$, with a total iteration step $k_T = 10^3$.}}
\label{table:test-pso}
\end{table}

\rev{
Table~\ref{table:test-pso} reports a comparison with the PSO method for the Ackley and Rastrigin functions. We adopt the formulation of PSO with also particles personal bests $\{ Y^i_k \}_{i=1,\dots}$ and inertia weight $m = 1-\gamma$, as described in \cite{grassiParticleSwarmOptimization2021, huangGlobalConvergenceParticle2023}. The main difference with respect to the standard PSO \cite{kennedyParticleSwarmOptimization1995} is the de-coupling of stochastic and deterministic components and Gaussian noise, which allows for more flexibility in balancing the explorative and exploitative behaviours of the dynamics.
%has shown superior performance, see e.g. \cite{borghiConsensusBasedOptimization2023
}
\rev{
The corresponding time-discrete update rule used in the comparison is given by
\begin{equation}
\begin{split}
    X^i_{k+1} &= X^i_k + \dt V^i_{k+1}, \\ 
    V^i_{k+1} &= \left( \frac{m}{m+\gamma \dt}\right) V^i_k + \frac{\lambda_1\dt}{m+\gamma \dt} (Y^i_k-X^i_k) + \frac{\lambda_2 \dt}{m+\gamma \dt} (\overline{Y}^\alpha_k - X^i_k) \\
    &\quad + \frac{\sigma_1 \sqrt{\dt}}{m+\gamma \dt}(Y^i_k-X^i_k) \odot \xi^{1,i}_k +  \frac{\sigma_2 \sqrt{\dt}}{m+\gamma \dt}(\overline{Y}^\alpha_k-X^i_k) \odot \xi^{2,i}_k, \\
    Y^i_{(k+1)\dt} &= 
    \begin{cases}
        Y^i_{k \dt}   & \textup{if}\;\; \EE(X^i_{(k+1)\dt}) \geq \EE(Y^i_{k\dt}) \\
        X^i_{(k+1) \dt} & \textup{if}\;\; \EE(X^i_{(k+1)\dt}) < \EE(Y^i_{k\dt}),
    \end{cases}
    \end{split}
    \label{eq:sdpso}
\end{equation}
where $\xi^{1,i}, \xi^{2,i} \sim \mathcal{N}(0,I_d)$. Here, $\overline{Y}^\alpha_k$ is a weighted average among personal bests, as in \eqref{eq:xalpha}, which for $\alpha\ll 1 $ almost coincides with the global best position up to time $k$.}

\rev{ To reduce the number of free parameters to the same number of the proposed one \eqref{eq:psojump}, we set $\lambda_1 = \lambda_2$ and $\sigma_1 = \sigma_2$, in line with \cite{grassiParticleSwarmOptimization2021}. The parameter choices reported in the table are justified by preliminary experiments: we tested this time-discrete formulation and identified the ranges of $\sigma_1, \sigma_2$ that yielded the best performance for both functions.
  Overall, both the proposed method and PSO achieve competitive results, but the use of Cauchy noise in our method leads to a slight improvement in performance on the Rastrigin function. We note that the possibility of accommodating Cauchy-type noise is a peculiarly of the BGK modelling framework proposed in this paper.
}

\rev{
\begin{remark}
The computational complexity in terms of number of operation per time-step is $\mathcal{O}(N)$, which is the minimal for the update of $N$ particles. The number of evaluation of objective function per iteration is also $N$, as in standard PSO and CBO algorithms. If the objective function is costly to evaluate, the 
weighted average point $X^\alpha[\rho_k^N]$ \eqref{eq:xalpha} can also be computed with 
only $M\ll N$ particles using random batch techniques \cite{carrilloConsensusbasedGlobalOptimization2020}. 
In terms of storage, the algorithm requires to store $N$ copies of $2d$-dimensional solutions ($2d$ and not simply $d$ due to the presence of velocities). Therefore, for very high-dimensional problems, only a small number of particles can be employed, possibly leading to poor performance. 
\end{remark}
}

\subsection{Comparison with CBO via scaling limit}
In Section~\ref{sec:2.scaling}, we showed that by taking the diffusive scaling limit $\e \to 0$ in the scaled BGK-type equation,
on can obtain the Fokker--Planck equation which characterised the mean-field description of the CBO particle dynamics. We investigate numerically in this section the relation between the two particle dynamics. 

We recall that, in the CBO method with anisotropic noise \cite{carrilloConsensusbasedGlobalOptimization2020}, the position $X^i_k$ of each particle $i=1, \dots, N$ for $k=0,1,\dots$ is updated according to:
\begin{equation}\label{eq:cbo}
X_{k+1}^i = X^i_k -\lambda \dt(X^i_k-X^\alpha_k[\rho_k^N])+\tilde{\sigma }\sqrt{\dt} (X^i_k-X^\alpha_k[\rho_k^N]) \odot \xi_k^i\, ,
\end{equation}
where we use a standard normal random variable $\xi_k^i \sim \mathcal{N}(0,I_d)$ as noise. 
We note that the discretization of the scaled update rule~\eqref{alg_BGK2} implemented here is not asymptotic preserving (AP), as it does not recover the CBO update \eqref{eq:cbo} in the limit $\e \to 0$. The development of an AP scheme remains an open problem. Therefore, at the algorithmic level, to obtain CBO behaviour as $\e \to 0$, we adopt the following scaling: for each $\tilde{\sigma}$ used in the CBO update rule, we set
\begin{equation}\label{eq:sigma-scale}
    \sigma = \frac{\tilde{\sigma}\e}{\sqrt{\dt}}
\end{equation} as the diffusion parameter in~\eqref{alg_BGK2}.

\begin{figure}
    \centering
    \includegraphics[width=\linewidth]{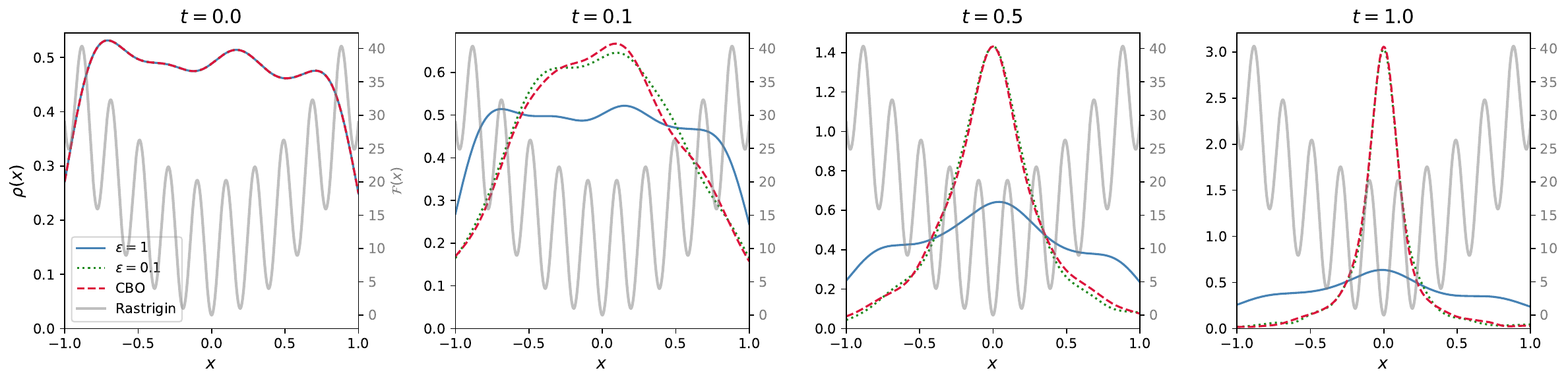}
    \caption{\rev{Density plot snapshots in the diffusion limit for the Rastrigin function in $d=1$, shown for rescaled jump PSO \eqref{alg_BGK2} with $\e = 1$ and $\e = 0.1$, alongside the CBO dynamics. The parameters are set to $N=1000$, $\tilde{\sigma}=1$ \text{ (for CBO)}, $\lambda=1$, $\dt = 0.1$, and $\alpha = 10^5$, with a total of $k_T= 10^3$ iteration steps. The diffusion constants for the scaled algorithm are chosen according to the scaling law \eqref{eq:sigma-scale}. Particles are visualised in the rescaled domain $[-1,1]$ and the density function is obtained via kernel density reconstruction. For reference, the objective function (Rastrigin) is plotted in the background with a separate axis on the right.  %The function \texttt{gaussian\_kde} from \texttt{scipy.stats} is used to visualise the particle density $\rho(x)$. 
    } 
    }
    \label{fig:density}
\end{figure}

\begin{figure}
    \centering
    \begin{subfigure}{0.48\textwidth}
        \centering
        \includegraphics[width=\textwidth]{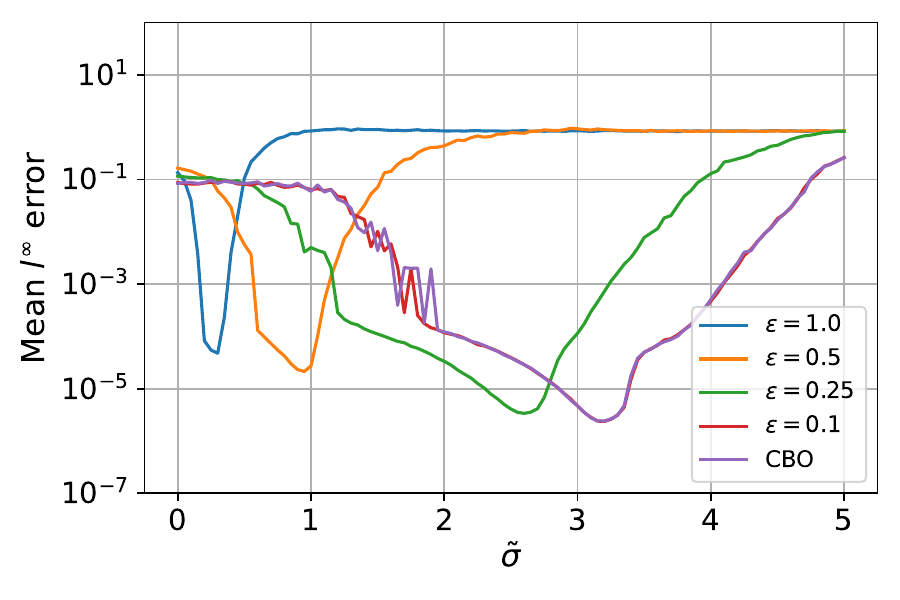}
        \caption{$\BE[\lVert X^\alpha_T-x^\ast\rVert_\infty]$, Ackley function}
        \label{fig:diff_ackley_mean}
    \end{subfigure}
    \hfill    
    \begin{subfigure}{0.48\textwidth}
        \centering
        \includegraphics[width=\textwidth]{ 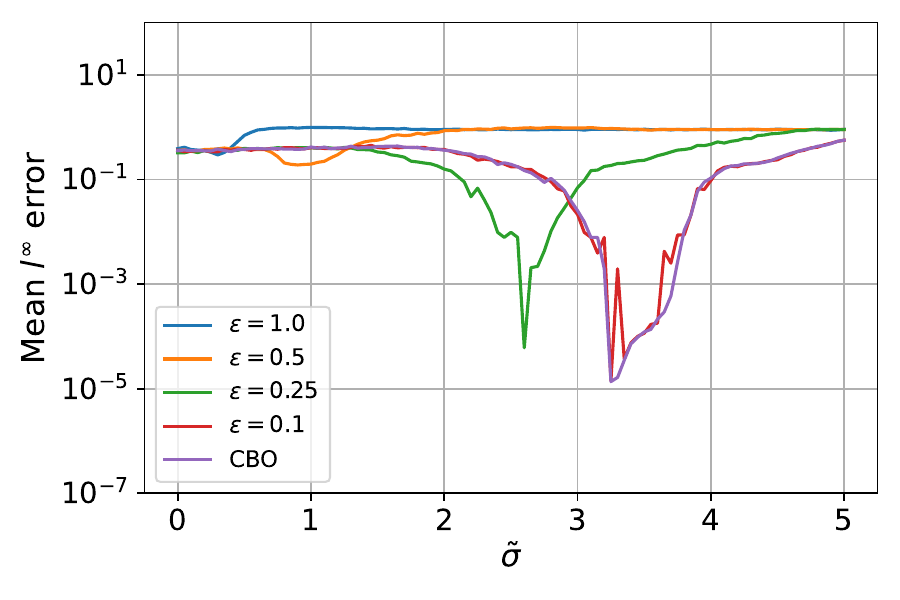}
        \caption{$\BE[\lVert X^\alpha_T-x^\ast\rVert_\infty]$, Rastrigin function}
        \label{fig:diff_rast_mean}
    \end{subfigure}
    \caption{Comparison with CBO via scaling limit. Tested on a range of $\tilde{\sigma}$ values from $0$ to $5$. For each $\e$ values, we plug in $\sigma=\tilde{\sigma}\e/\sqrt{\dt}$ in our algorithm. Ackley (left) and Rastrigin (right) functions are used in $d=20$, parameters are set to $N=200, \dt=0.1, \lambda=1, \nu=1, \alpha=10^5$, with total iteration step $k_T=10^3$.}
    \label{fig:diff} 
\end{figure}

\rev{We begin with an illustrative experiment. In Figure~\ref{fig:density} we can observe the evolution of the particle density in a $d=1$ experiment for different values of $\e$. We note that for $\e=0.1$ the scaled algorithm~\eqref{alg_BGK2} behaves as the CBO one, as expected, while larger values of $\e$ lead to a more spread distribution around the minimizer.
}

We conduct experiments on the Ackley and Rastrigin functions across a range of $\e$ values, $\e = 1, 0.5, 0.25, 0.1$, to \rev{illustrate} the convergence behaviour as $\e$ approaches zero \rev{in $d=20$}.  
\rev{Figure~\ref{fig:diff} shows the expected consistency in the limit, also in higher dimensions, between the two functions.}
As the scaling parameter $\e$ approaches zero, the error dynamics of \rev{\eqref{alg_BGK2}} progressively converge to those exhibited by the CBO algorithm. Notably, for both the Ackley and Rastrigin functions, we observe that from $\e = 0.1$, the error dynamics become indistinguishable from those of the CBO algorithm, confirming our theoretical expectation regarding the limiting behaviour of the algorithm.

%%%%%%%%%%%%%%%%%%%%%%%%%%%%%%%%%%%%%%%%%%%%%%%%%%%%%%%%%%%%%%%%%%%%%%%%%%%%
\section{Concluding remarks}
\label{sec:conclusion}
In this work, we proposed and analysed a novel swarm-based optimization algorithm in which the velocity update is performed via stochastic jumps. This strategy, inspired by existing PSO extensions and mutation-based metaheuristics, aims to improve exploration and prevent premature convergence, while remaining amenable to mathematical modelling. Starting from the discrete particle system, we derived a kinetic description of BGK type that captures the collective behaviour of particles in the large-particles limit. Unlike standard PSO models relying on Gaussian noise, our framework allows for general noise distributions and provides a unified view that includes heavy-tailed perturbations, such as those generated by Cauchy variables. Under a suitable diffusive rescaling, we derive the connection between the proposed model and Consensus-Based Optimization (CBO), showing that the macroscopic density satisfies a nonlinear Fokker–Planck equation of CBO type. 

Within this framework, we proved quantitative propagation of chaos for the particle system in bounded domains, demonstrating that its empirical measure converges to the kinetic solution with explicit error bounds. Moreover, we established convergence towards global minimizers of the objective function in the  setting of non-degenerate diffusion and convex search domains. To our knowledge, this is the first convergence result of this type for second-order swarm-based optimization methods.
The theoretical results have been complemented by numerical experiments on classical benchmark problems, validating both the efficiency of the proposed algorithm and the predicted scaling behaviour. The simulations also confirmed the role of the diffusion parameter and jump frequency in determining convergence rates and robustness to noise. Furthermore, consistency with the CBO algorithm in the diffusive limit was numerically verified.

Several directions for future research remain open. From a theoretical viewpoint, an extension of the convergence analysis to unbounded domains, as well as to more general noise models, would further broaden the applicability of the framework. Another important line of investigation concerns the extension of the propagation of chaos and convergence results to the classical PSO dynamics, which include local memory terms and inertia weights. The proof techniques developed here may serve as a foundation for this more general analysis. Moreover, while we formally derived the connection between the BGK dynamics and CBO through a diffusive scaling, a rigorous mathematical justification of the limit remains an open problem, as well as the development of a consistent asymptotic preserving particle scheme. Proving convergence from the kinetic BGK model to the macroscopic CBO equation would provide a solid foundation for the observed asymptotic behaviour and clarify the role of noise structure in the limiting dynamics \cite{golse2014fluid}.

%%%%%%%%%%%%%%%%%%%%%%%%%%%%%%%%%%%%%%%%%%%%%%%%%%%%%%%%%%%%%%%%%%%%%%%%%%%%
\section*{Acknowledgment}
The research has been supported by the Royal Society under the Wolfson Fellowship ``Uncertainty quantification, data-driven simulations and learning of multiscale complex systems governed by PDEs". 
The partial support by the European Union through the Future Artificial Intelligence Research (FAIR) Foundation, ``MATH4AI" Project and by ICSC -- Centro Nazionale di Ricerca in High Performance Computing, Big Data and Quantum Computing, funded by European Union -- NextGenerationEU and by the Italian Ministry of University and Research (MUR) through the PRIN 2022 project (No. 2022KKJP4X) ``Advanced numerical methods for time dependent parametric partial differential equations with applications" is also acknowledged. This work has been written within the activities of GNFM and GNCS groups of INdAM (Italian National Institute of High Mathematics).
\bibliographystyle{siam}
\bibliography{biblio}
\end{document}